\documentclass{article}
\usepackage{fullpage}
\usepackage{enumitem}
\usepackage{xcolor}
\usepackage{amsmath}

\usepackage{amsthm}
\usepackage{mathtools}
\usepackage[english]{babel}
\usepackage{amsxtra}
\usepackage{lipsum}
\usepackage{amsfonts}
\usepackage{graphicx}
\usepackage{subfig}
\usepackage{url}
\usepackage{epstopdf}
\usepackage{algorithmic}
\usepackage{caption}
\usepackage{float}
\newtheorem{theorem}{Theorem}
\newtheorem{lemma}{Lemma}
\newtheorem{remark}{Remark}

\DeclareMathOperator{\sign}{sign}
\DeclareMathOperator{\Diag}{Diag}

\newtheorem{myassump}{Assumption}
\def\Tr{\mathop{\rm \boldmath Tr}}
\def\argmin{\mathop{\rm \boldmath argmin}}

\begin{document}
	
	\begin{center}
		{\bf {\LARGE{Nonlinear consensus+innovations under correlated heavy-tailed noises: Mean square convergence rate and asymptotics }}}
		
		\vspace*{.2in}
		
				\large{
					\begin{tabular}{cc}
						 Manojlo Vukovic$^\circ$, & Dusan Jakovetic$^\ast$,
					\end{tabular}
					\begin{tabular}{ccc}
						Dragana Bajovic$^\dagger$, & Soummya Kar$^\rhd$
					\end{tabular}
				}
			
		\vspace*{.2in}
			
		\begin{tabular}{c}
        $^\circ$University of Novi Sad, Faculty of Technical Sciences, Department of\\ Fundamental Sciences (manojlo.vukovic@uns.ac.rs)\\
		$^\ast$ University of Novi Sad, Faculty of Sciences, Department of\\ Mathematics and Informatics (dusan.jakovetic@dmi.uns.ac.rs)
		\\ $^\dagger$ University of Novi Sad, Faculty of Technical Sciences, Department of \\Power, Electronic and
		Communication Engineering(dbajovic@uns.ac.rs) \\ 
		$^\rhd$ Department of Electrical and Computer Engineering, \\Carnegie Mellon University (soummyak@andrew.cmu.edu)
		\\
		\end{tabular}
				
				\vspace*{.2in}
				
				\today
				
				\vspace*{.2in}
		
	\end{center}
	
	\begin{abstract}
	We consider distributed recursive estimation of consensus+innovations type
	 in the presence of heavy-tailed sensing and communication noises. 
	 We allow that the sensing and communication noises are mutually correlated while independent identically distributed (i.i.d.) in time, and that they may both have infinite moments of order higher than one (hence having infinite variances). Such heavy-tailed, infinite-variance noises are highly relevant in practice and are shown to occur, e.g., in dense internet of things (IoT) deployments. We develop a consensus+innovations distributed estimator that employs a general nonlinearity in both consensus and innovations steps to combat the noise. We establish the estimator's almost sure convergence, asymptotic normality, and mean squared error (MSE) convergence. Moreover, we establish and explicitly quantify for the estimator a sublinear MSE convergence rate.  
	  We then quantify through analytical examples the effects of the 
	 nonlinearity choices  and the noises correlation on the system performance. Finally, 
	 numerical examples corroborate our findings and verify that the proposed method 
	 works in the simultaneous heavy-tail communication-sensing noise setting, while existing methods fail under the same noise conditions.
	\end{abstract}
	
	\section{Introduction}
	We consider a distributed estimation problem where a network of 
	agents cooperates to estimate an unknown static vector parameter  $\boldsymbol{\theta}^\ast \in {\mathbb R}^M$. Specifically, we are interested in \emph{consensus+innovations} distributed estimation, e.g.,~\cite{KMR,SoummyaAsymptEfficient,SoummyaAdaptive}. With consensus+innovations,  each agent iteratively updates its unknown parameter's estimate by 
	1) exchanging its estimate with immediate neighbors in the network; and 2) assimilating a newly acquired observation (measurement). 
	
	Consensus+innovations distributed estimators have been extensively studied, e.g., \cite{KMR,SoummyaAsymptEfficient,SoummyaAdaptive}; see also \cite{A1,A2,A3,A4,A5,E,G1} for related diffusion-type and other methods.  
	Typically, such distributed estimators exhibit strong convergence guarantees  under various imperfection models (noises) in 1) sensing (observations) and/or 2) inter-agent communications. For example, reference~\cite{KMR} establishes almost sure (a.s.)  convergence and asymptotic normality of the estimators developed therein.  The authors of~\cite{KMR} allow for an observation noise with finite variance and a network model that accounts for random link failures and dithered quantization (effectively an additive noise with finite variance).  Reference~\cite{SoummyaAsymptEfficient} considers consensus+innovations distributed estimation in the presence of random link failures without quantization or additive noise, and it develops estimators that are asymptotically efficient, i.e., that achieve the minimal possible asymptotic variance. The authors of~\cite{SoummyaAdaptive} propose adaptive asymptotically efficient estimators, wherein the innovation gains are adaptively learned during the algorithm progress. 
	Consensus+innovations distributed detection and related distributed detection methods have also been considered, e.g.,~\cite{SayedDetection,D1,D2,D3}. The above distributed estimation and distributed detection-related works typically assume that the noises have finite moments of a certain order greater than two, and hence they have finite variance.  
	
	It is highly relevant to investigate distributed estimators in the presence of heavy-tailed \emph{communication} and \emph{sensing} noises,  as they arise in many application scenarios. For example, edge devices in Internet of Things (IoT) systems or sensor networks can be subject to noise distributions that may not have finite moments of order higher than one, e.g.,~\cite{IoTHeavyTail,IoTHeavyTail2,18,19,20,DasNonlinConsensus}, like, e.g., symmetric $\alpha$-stable noise distributions. 
    {This effect may occur due to interference, e.g., when wireless sensor network is relatively densely deployed. In this case, the signals of neighboring nodes interfere with each other and corrupt the signal to be received.
    References~\cite{Haenggi,Win} analyze the probability distribution of the interference and demonstrate that it has heavy-tails. More precisely,  \cite{Haenggi,Win} show that the interference power has an alpha-stable distribution in a network with infinite radius and no guard zone when the interferers are placed according to a Poisson point process, where alpha depends on the path loss coefficient between the interferers and the receiver (see~\cite{Haenggi,Win} for details). Empirical evidence for the emergence of heavy-tail interference noise in certain IoT systems has been provided in \cite{IoTHeavyTail}.}

    Moreover, observation and communication noises may be mutually correlated due to the common interference processes in the environment that the sensing and communication devices are exposed to. 
	
	Several recent works~\cite{R2,R3,R5,I,G,F,A,R1} consider distributed estimation methods in the presence of \emph{impulsive observations noise},\footnote{As explained in, e.g., \cite{F}, an
impulsive noise may be described as one whose realizations contain sparse, random samples of amplitude
much higher than nominally accounted for. Impulsive noise may have a finite or infinite variance. Existing works on distributed estimation in impulsive noises assume a \emph{finite noise variance}.} but still assuming a \emph{finite noise variance} and \emph{no communication noise}. 
	For example, reference~\cite{R2} introduces a method based on Wilcoxon-norm;~\cite{R3} utilizes a Huber-loss function; and~\cite{R5} adopts a mean error minimization approach. 
	Robust distributed estimation methods based on adaptive subgradient projections are considered in~\cite{I,G}. 
	To cope with the impulsive observation noise, several references employ a certain \emph{nonlinearity} in the innovation step. 
	Reference~\cite{F} develops a method that adaptively learns an optimized nonlinearity at the innovation step for each agent in the network.  Reference~\cite{A} employs a saturation nonlinearity in the innovation step to cope with measurement attacks. Further results on distributed estimation under impulsive observations noise can be found in a recent survey~\cite{R1}. 
	Very recently, we have developed a consensus+innovations distributed estimator~\cite{Ourwork} that provably works under 
	a heavy-tailed communications noise and a light-tailed 
	observations noise. Specifically, under the assumed setting, \cite{Ourwork} 
	establishes almost sure convergence and asymptotic normality of the method therein.
	 However, \cite{Ourwork} is not concerned with mean squared error (MSE) rate analysis 
	 of the method. While asymptotic normality is a useful result that provides 
	 the algorithm's rate of convergence (in the weak convergence sense)
 \emph{asymptotically}, it does not capture the (MSE) algorithm behavior in non-asymptotic regimes.	
	
	In summary, we identify for the current literature the following major gaps with respect to design and analysis of distributed estimation methods under heavy-tailed noises. 1) All existing works assume a finite observations noise variance. That is, even when impulsive observation noise is assumed, existing works still require the variance of the noise to be finite. This assumption can be restrictive and is violated for several commonly used heavy-tail noise models like $\alpha$-stable distributions~\cite{20}. 2) No existing work simultaneously handles 
	heavy-tailed (infinite-variance) sensing and heavy-tailed (infinite-variance) observation noises. 3) 
	MSE convergence rate analysis has not been developed for distributed estimation in the presence of either infinite-variance sensing and/or infinite-variance communication noises. 4) Existing works on distributed estimation in 
	the presence of infinite-variance (either sensing and/or communication noises) assume mutually independent sensing and communication noises.
	
	\textbf{Contributions}. In this paper, we close the gaps identified above 
	by developing a nonlinear consensus+innovations distributed estimator that provably works under the simultaneous presence of correlated heavy-tailed (infinite variance) observation and communication noises. We allow for a very general model of the sensing and communication noises, only assuming that 
	they exhibit symmetric zero-mean distributions with finite first moments. Hence, the variances of both sensing and 
	communication noises may be infinite. Moreover, we allow that, for a fixed time instant $t$, the additive sensing and communication noises may be mutually dependent, while they are both independent identically distributed (i.i.d.) in time. The proposed estimator employs a generic nonlinearity both at the innovations and the consensus terms.  
	The encompassed nonlinearities are very general and include a broad class of (possibly discontinuous) odd functions, such as the component-wise sign and clipping functions.   
	We establish for the proposed estimator almost sure convergence, asymptotic normality, and we explicitly evaluate the corresponding asymptotic variance.  Furthermore, we establish 
	for the proposed method, under a carefully designed step size sequence, a MSE convergence rate $O(1/t^\kappa)$, and we quantify the rate $\kappa \in (0,1)$ in terms of the system parameters. 
	In addition, we quantify through analytical examples the effects of correlation between sensing and observation noises, and we demonstrate how the derived asymptotic covariance results may be used as a guideline to optimize the employed nonlinearities for a problem at hand. 
	 Finally, we compare the proposed method with existing works in~\cite{F} and~\cite{Ourwork}, both through analytical examples and by simulation. Most notably, we show that the existing methods fail to converge under the simultaneous presence of heavy-tailed (infinite-variance) observation and communication noises, while the proposed method provably works in the heavy-tailed setting. 
	
	\textbf{Paper organization}.
    Section~\ref{section-model-algorithm} provides a description of the distributed estimation model that is considered and also gives all basic assumptions. In Section~\ref{section-proposed-algorithm}, we present the proposed nonlinear consensus+innovations estimator. Section \ref{section-theoresults} establishes almost sure convergence, asymptotic normality and the MSE rate of the proposed distributed estimator. Section~\ref{section:examples} presents analytical and numerical examples. The conclusion is given in Section~\ref{section-conlusion}. Some auxiliary supporting arguments are provided in Appendix.
 
	\textbf{Notation}. We denote by $\mathbb R$ the set of real numbers and by ${\mathbb R}^m$ the $m$-dimensional
	Euclidean real coordinate space. We use normal lower-case letters for scalars,
	lower case boldface letters for vectors, and upper case boldface letters for
	matrices. Further,  to represent a vector $\mathbf{a}\in\mathbb{R}^m$ through its component, we write $\mathbf{a}=[\mathbf{a}_1, \mathbf{a}_2, ...,\mathbf{a}_m]^\top$ and we denote by: $\mathbf{a}_i$ or $[\mathbf{a}_i]$, as appropriate, the $i$-th element of vector $\mathbf{a}$; $\mathbf{A}_{ij}$ or $[\mathbf{A}_{ij}]$, as appropriate, the entry in the $i$-th row and $j$-th column of
	a matrix $\mathbf{A}$;
	$\mathbf{A}^\top$ the transpose of a matrix $\mathbf{A}$; $\otimes$ the Kronecker product of matrices. Further, we use either  
	$\mathbf{a}^\top \mathbf{b}$ or 
	$\langle \mathbf{a},\,\mathbf{b}\rangle$ 
	for the inner products of vectors 
	$\mathbf{a}$ and $\mathbf{b}$. Next, we let  
	$\mathbf{I}$, $\mathbf{0}$, and $\mathbf{1}$ be, respectively, the identity matrix, the zero vector, and the column vector with unit entries; $\Diag(\mathbf{a})$ the diagonal matrix 
	whose diagonal entries are the elements of vector~$\mathbf{a}$; 
	$\mathbf{J}$ the $N \times N$ matrix $\mathbf{J}:=(1/N)\mathbf{1}\mathbf{1}^\top$.
	When appropriate, we indicate the matrix or vector dimension through a subscript.
	Next, $\mathbf{A}\succ  0 \,(\mathbf{A} \succeq  0 )$ means that
	the symmetric matrix $A$ is positive definite (respectively, positive semi-definite).
	We further denote by:
	$\|\cdot\|=\|\cdot\|_2$ the Euclidean (respectively, spectral) norm of its vector (respectively, matrix) argument; $\lambda_i(\cdot)$ the $i$-th smallest eigenvalue; $g^\prime(v)$ the derivative evaluated at $v$ of a function $g:\mathbb{R}\to\mathbb{R}$; $\nabla h(\mathbf{w})$ and $\nabla^2 h(\mathbf{w})$ the gradient and Hessian, respectively, evaluated at $w$ of a function $h: {\mathbb R}^m \rightarrow {\mathbb R}$, $m > 1$; $\mathbb P(\mathcal A)$ and $\mathbb E[u]$ the probability of
	an event $\mathcal A$ and expectation of a random variable $u$, respectively; and by $\sign(a)$ the sign function, i.e., $\sign(a)=1$, for $a>0$, $\sign(a)=-1$, for $a<0$, and $\sign(0)=0$.
	Finally, for two positive sequences $\eta_n$ and $\chi_n$, we have: $\eta_n = O(\chi_n)$ if
	$\limsup_{n \rightarrow \infty}\frac{\eta_n}{\chi_n}<\infty$.

	\section{Problem model and basic assumptions}
	\label{section-model-algorithm}
	
	We consider a network of $N$ agents (sensors), through which the parameter of interest $\boldsymbol{\theta}^{\ast}\in\mathbb{R}^{M}$ is to be estimated. At each time $t=0,1,...,$ each agent $i=1,2,...,N$ observes parameter $\boldsymbol{\theta}^{\ast}$ following the linear regression model:
	\begin{align}
		\label{eq:obs_model}
		z_{i}^{t} = \mathbf{h}_{i}^{\top}\boldsymbol{\theta}^{\ast}+n_{i}^{t}.
	\end{align}
	Here, $z_{i}^{t}\in\mathbb{R}$ is the observation, $\mathbf{h}_{i}\in\mathbb{R}^{M}$ is the deterministic, non-zero regression vector known only by agent $i$ and $n_{i}^{t}\in\mathbb{R}$ is the observation noise. The underlying topology is modeled via a graph $G=(V,E),$ where $V=\{1,...,N\}$ is the set of agents and $E$ is the set of links, i.e., $\{i,j\}\in E$ if there exists a link between agents $i$ and $j$. We also define the set of all arcs $E_d$ in the following way: if $\{i,j\}\in E$ then $(i,j)\in E_d$ and $(j,i)\in E_d$.
	We denote by $\Omega_i=\{j\in V: \{i,j\}\in E\}$ set of neighbors of agent~$i$ (excluding $i$) and by $\mathbf{D}=\Diag(\{d_i\})$ the degree matrix, where $d_i=|\Omega_i|$ is the number of neighbors of agent~$i$. The graph Laplacian matrix $\mathbf{L}$ is defined by $\mathbf{L}=\mathbf{D}-\mathbf{A}$, where $\mathbf{A}$ is the adjacency matrix, which is a zero-one symmetric matrix with zero diagonal, such that, for $i\neq j$, $\mathbf{A}_{ij}=1$ if and only if $\{i,j\}\in E.$ 
	Let us denote by $(\Omega,\mathcal{F},\mathbb{P})$ the underlying probability space.
	
	\noindent We make the following assumptions.
	
	\begin{myassump}\label{as:newtworkmodelandobservability}\textbf{Network model and Observability:} 
		\begin{enumerate}
			\item Graph $G=(V,E)$ is undirected, simple (no self or multiple links) and static;
			\item The matrix $\sum_{i=1}^{N}\mathbf{h}_i\mathbf{h}_i^{\top}$ is invertible;
		\end{enumerate}
	\end{myassump}
	
	\noindent The condition 2 in Assumption~\ref{as:newtworkmodelandobservability} ensures that \eqref{eq:obs_model} is observable, i.e., a centralized estimator (e.g., least squares) that collects all $z_i^t, i=1,2,...,N,$ for all $t$, and has knowledge of all vectors $\mathbf{h}_i, i=1,2,...,N,$ is consistent.
	
	\begin{myassump}\label{as:observationnoise}\textbf{Observation noise:}
		\begin{enumerate}	
			\item For each agent 
			$i=1,...,N$, the observation noise sequence $\{{n}_{i}^{t}\}$ in \eqref{eq:obs_model},
			is independent identically distributed (i.i.d.);
            \item {At each agent $i=1,...,N$~at each time~$t=0,1,...,$ noise $n_i^t$ has the same probability density function~$p_\mathrm{o}$.}
			\item Random variables ${n}_{i}^{t}$ and ${n}_{j}^{s}$ 
			are mutually independent whenever the tuple  $(i,t)$ 
			is different from $(j,s)$;
            \item {The pdf $p_\mathrm{o}$ is symmetric, i.e. $p_\mathrm{o}(u)=p_\mathrm{o}(-u),$ for every $u\in\mathbb{R}$, and $p_\mathrm{o}(u)>0$ for $|u|\leq c_\mathrm{o}$, for some constant $c_\mathrm{o}>0$;}
            \item {There holds that with $\int |u|p_\mathrm{o}(u)du <\infty$.}
		\end{enumerate}
	\end{myassump}
	
	\noindent If there is an arc between agents $i$ and $j$, i.e., $(i,j)\in E_d$, we denote by $\boldsymbol{\xi}_{ij}^t$ communication noise that is injected when agent $j$ communicates to agent $i$ at time instant $t$ (see ahead algorithm~\eqref{eq:alg2}).
	
	\begin{myassump}\label{as:communicationnoise}\textbf{Communication noise:}
		\begin{enumerate}
			\item Additive communication noise $\{\boldsymbol{\xi}^t_{ij}\}$, $\boldsymbol{\xi}^t_{ij} \in \mathbb{R}^M$ is i.i.d. in time $t$, and independent across different arcs~$(i,j)\in E_d.$
            \item {Each random variable $[\boldsymbol{\xi}^t_{ij}]_{\ell}$, for 
			each $t=0,1...$, for each arc $(i,j)$, 
			for each entry $\ell=1,...,M$, 
			has the same probability density function~$p_\mathrm{c}$.}
            \item {The pdf $p_\mathrm{c}$ is symmetric, i.e. $p_\mathrm{o}(u)=p_\mathrm{c}(-u),$ for every $u\in\mathbb{R}$ and $p_\mathrm{c}(u)>0$ for $|u|\leq c_\mathrm{c}$, for some constant $c_\mathrm{c}>0$;}
            \item {There holds that $\int |u|p_\mathrm{c}(u)du <\infty$.}
		\end{enumerate}
	\end{myassump}
	\begin{remark}
	{Notice here that from the symmetry of the probability density functions $p_{\mathrm{o}}$ and $p_{\mathrm{c}}$}, it follows that both of the distributions are zero mean. Moreover, notice that we do not assume that observation and communication noises are mutually independent for a fixed $t$. However, they are both i.i.d. in time.
    \end{remark}
	
    \begin{remark}
    Condition 2 in Assumptions~\ref{as:observationnoise} and~\ref{as:communicationnoise} can be relaxed in the sense that it can be assumed that $\mathbf{n}^t$ has joint probability density function $p_\mathrm{o}$ and $\boldsymbol{\xi}_{ij}^t$ has the joint probability density function $p_{\mathrm{c},ij}.$ (see Appendix C).
{The reason why there is condition 4 in the Assumption~\ref{as:observationnoise} and condition 3 in the Assumption~\ref{as:communicationnoise}  will become clear later.}
\end{remark}
	
	\noindent For future reference, a compact vector form of~\eqref{eq:obs_model} is:
	\begin{align}
		\label{eq:obs_vec}
		\mathbf{z}^{t} = \mathbf{H}\left(\mathbf{1}_{N}\otimes \boldsymbol{\theta}^{\ast}\right)+\mathbf{n}^{t},
	\end{align}
	where, $\mathbf{z}^t=[z_1^t,z_2^t,...,z_N^t]^\top\in\mathbb{R}^N$ is the observation vector, $\mathbf{H}\in\mathbb{R}^{N\times(MN)}$ is the regression matrix whose $i$-th row vector equals $[\mathbf{0},...,\mathbf{0},\mathbf{h}_i^\top,\mathbf{0},..,\mathbf{0}]\in\mathbb{R}^{MN}$, where the $i$-th block of size $M$ equals $\mathbf{h}_i^\top$, and the other $M$-size blocks are the zero vectors;  and $\mathbf{n}^t=[n_1^t,{n}_2^t,...,{n}_N^t]^\top\in\mathbb{R}^N$ is the noise vector at time $t$.

	\section{Proposed algorithm}\label{section-proposed-algorithm}
	
	In order to estimate the unknown parameter $\boldsymbol{\theta}^\ast \in {\mathbb R}^M$, in the presence of heavy-tailed observation noise and heavy-tailed communication noise, each agent uses a nonlinear consensus+innovations strategy. Therein, the impact of the two heavy-tailed noises is mitigated by nonlinearities that have been added to both consensus and innovation steps.
	
	\noindent In more detail, each agent $i$ at each time $t=0,1,...,$  generates a sequence of estimates $\{\mathbf{x}_{i}^{t}\}_{t\geq0}$ of unknown parameter $\boldsymbol{\theta}^\ast$ by the following algorithm:
	
	\begin{align}\label{eq:alg2}
		\mathbf{x}_{i}^{t+1}=\mathbf{x}_{i}^{t}-\alpha_{t}\left(\frac{b}{a}\sum_{j\in\Omega_{i}}
		\boldsymbol{\Psi}_\mathrm{c}\left( \mathbf{x}_{i}^{t}-\mathbf{x}_{j}^{t} 
		+\boldsymbol{\xi}_{ij}^t\right)-\mathbf{h}_{i}{\Psi}_\mathrm{o}\left(z_{i}^{t}-\mathbf{h}_{i}^{\top}\mathbf{x}_{i}^{t}\right)\right).
	\end{align} %
	Here, $\alpha_t$ is a step-size, 
	and $a,b>0$ are constants. We consider a family of decaying step-size choices $\alpha_t= a/(t+1)^\delta,$ $\delta\in (0.5,1].$ As shown later, the step-size (values of $a$ and $\delta$) should be designed appropriately in order for good properties (e.g., a.s. convergence, MSE rate guarantees) of the algorithm to hold. Functions ${\Psi}_\mathrm{o}:\mathbb{R}\to\mathbb{R}$ and $\boldsymbol{\Psi}_\mathrm{c}:  \mathbb{R}^M\to\mathbb{R}^M$ are non-linear functions and function $\boldsymbol{\Psi}_\mathrm{c}$ operates component-wise by abusing notation, i.e., for
	 $\mathbf{y}\in\mathbb{R}^M$, we set that  $\boldsymbol{\Psi}_\mathrm{c}(\mathbf{y})=[{\Psi}_\mathrm{c}(\mathbf{y}_1),{\Psi}_\mathrm{c}(\mathbf{y}_2),...,{\Psi}_\mathrm{c}(\mathbf{y}_M)].$ Also, functions ${\Psi}_\mathrm{c}$ and ${\Psi}_\mathrm{o}$ satisfy Assumption~\ref{as:nonlinearity}. We compare the proposed method \eqref{eq:alg2} with the $\mathcal{LU}$ scheme in~\cite{KMR} and the scheme in~\cite{Ourwork}. Compared with these schemes,~\eqref{eq:alg2} introduces a nonlinearity in the innovation step as well. $\mathcal{LU}$ is obtained from~\eqref{eq:alg2} by setting both of the nonlinearities $\Psi_{\mathrm{o}}$ and $\Psi_{\mathrm{o}}$ to identity functions and $\delta=1$, the method in~\cite{Ourwork} is recovered from~\eqref{eq:alg2} by setting $\Psi_{\mathrm{o}}$ to the identity function and $\delta=1$.
	 	
	\begin{myassump}\label{as:nonlinearity}\textbf{Nonlinearity $\Psi$:}\\
		The non-linear function $\Psi:\mathbb{R}\to\mathbb{R}$ satisfies the following properties:
		\begin{enumerate}
			\item Function~$\Psi$ is odd, i.e., $\Psi(a)=-\Psi(-a),$ for any $a\in\mathbb{R}$;
			\item $\Psi(a) > 0,$ for any $a > 0$.
			\item Function $\Psi$ is a monotonically nondecreasing function;
			\item $\Psi$ is continuous, except possibly on a point set with Lebesque measure of zero. Moreover, $\Psi$ is piecewise differentiable;
            \item {$|\Psi(a)|\leq c_1$, for some constant $c_1>0.$}
            \item { $\Psi$ is either discontinuous at zero, or $\Psi(u)$ is strictly increasing for $u\in(-c_2,c_2)$, for some $c_2 > 0.$}
		\end{enumerate}
	\end{myassump}
	
	\noindent As it will become clear ahead, the role of $\Psi_\mathrm{c}$ and $\Psi_\mathrm{o}$ is to lower the impact of the heavy-tailed noise that occurs in the regression model and in the communication between agents. As it is presented in~\cite{Ourwork}, there are many nonlinear functions which satisfy Assumption~\ref{as:nonlinearity}. Now, we add more assumptions on the observation and communication noises through the following assumption.

	\noindent At each time $t=0,1,...$, a compact vector form of algorithm~\eqref{eq:alg2} is
	\begin{align}\label{eq:algcomp}
		\mathbf{x}^{t+1}=\mathbf{x}^{t}-\alpha_t\left( \frac{b}{a}\mathbf{L}_{\boldsymbol{\Psi}_\mathrm{c}}(\mathbf{x})- \mathbf{H}^{\top}\boldsymbol{\Psi}_\mathrm{o}\left(\mathbf{z}^{t}-\mathbf{H}\mathbf{x}^{t}\right)\right).
	\end{align}
	Here, $\mathbf{x}^t=[\mathbf{x}_1^t,\mathbf{x}_2^t,...,\mathbf{x}_N^t]^\top\in\mathbb{R}^{MN},$ map $\mathbf{L}_{\boldsymbol{\Psi}_\mathrm{c}}(\mathbf{x}):\mathbb{R}^{MN}\to\mathbb{R}^{MN}$ is defined by
	\begin{align*}
		\mathbf{L}_{\boldsymbol{\Psi}_\mathrm{c}}(\mathbf{x})=\begin{bmatrix}
			\vdots\\
			\sum\limits_{j\in\Omega_{i}}\boldsymbol{\Psi}_\mathrm{c}(\mathbf{x}_i-\mathbf{x}_j+\boldsymbol{\xi}_{ij})\\
			\vdots
		\end{bmatrix},
	\end{align*}
	where, the blocks $\sum\limits_{j\in\Omega_{i}}\boldsymbol{\Psi}_\mathrm{c}(\mathbf{x}_i-\mathbf{x}_j+\boldsymbol{\xi}_{ij})\in\mathbb{R}^M$ are stacked one on top of another for $i=1,...,N.$ 
	
	\section{Theoretical results}\label{section-theoresults}
	
	In subsection~\ref{subsection:rewritingalg} we express algorithm~\eqref{eq:algcomp} in more general way, that will be used in the following subsections. Subsection~\ref{subsection:as} presents the statement and the proof of almost sure convergence of algorithm~\eqref{eq:alg2}. In subsection~\ref{subsection:an} we state and prove asymptotic normality and calculate the corresponding asymptotic variance. Subection~\ref{subsection-MSE} presents and proves results on MSE rates.
	
	\subsection{Setting up analysis}
	\label{subsection:rewritingalg}
	
	In this subsection we rewrite algorithm~\eqref{eq:alg2} in the form suitable for stating the main results. To do that, firstly we define function $\varphi:\mathbb{R}\to\mathbb{R}$ by
    {
	\begin{align}
		\label{eq:phi-def}
		\varphi(a)=\int \Psi(a+w)p(w) dw,
	\end{align}}
    {where $\Psi:\mathbb{R}\to\mathbb{R}$ is a nonlinear function that satisfies Assumption~\ref{as:nonlinearity}, and $p$ is a probability density function that satisfies Assumptions~\ref{as:observationnoise} or~\ref{as:communicationnoise}.}
    \begin{remark}
        {The mapping $\varphi$ has all key properties of function $\Psi$ (see Lemma 6 in Appendix B, see also~\cite{PolyakNL}). Moreover, it has a strictly positive derivative at zero, i.e., $\varphi^\prime(0)>0$, which is necessary to prove our results. The facts that the nonlinearity $\Psi$ is discontinuous at zero or that it has a positive derivative at zero, together with
        condition 4  from Assumptions~\ref{as:observationnoise} and condition 3 from~\ref{as:communicationnoise}, are crucial to ensure that $\varphi$ has a positive derivative at zero (see Appendix B, see also~\cite{Ourwork,PolyakNL}). Notice that the requirement that the pdf $p$ is positive in the vicinity of the zero is not restrictive, since it holds true for a broad classes of non-zero noise pdfs.}
    \end{remark}
    Next, we define functions $\boldsymbol{\varphi}_\mathrm{o}:\mathbb{R}^{N}\to\mathbb{R}^{N}$, $\boldsymbol{\varphi}_\mathrm{c}:\mathbb{R}^{M}\to\mathbb{R}^{M}$ as $\boldsymbol{\varphi}_\mathrm{o}(\mathbf{y}_1,\mathbf{y}_2,...,\mathbf{y}_{N})=[\varphi_\mathrm{o}(\mathbf{y}_1),\varphi_\mathrm{o}(\mathbf{y}_2),...,\varphi_\mathrm{o}(\mathbf{y}_{N})]$, $\boldsymbol{\varphi}_\mathrm{c}(\hat{\mathbf{y}}_1,\hat{\mathbf{y}}_2,...,\hat{\mathbf{y}}_{M})=[\varphi_\mathrm{c}(\hat{\mathbf{y}}_1),\varphi_\mathrm{c}(\hat{\mathbf{y}}_2),...,\varphi_\mathrm{c}(\hat{\mathbf{y}}_{M})]$, where $\mathbf{y}\in\mathbb{R}^{N}$, $\hat{\mathbf{y}}\in\mathbb{R}^{M}$ and functions $\varphi_\mathrm{o}$ and $\varphi_\mathrm{c}$ are transformations defined by~\eqref{eq:phi-def} that correspond to $\Psi_\mathrm{o}$ and $\Psi_\mathrm{c}$, respectively. For the a.s. convergence and asymptotic normality results, we will follow the stochastic approximation framework from~\cite{Nevelson,KMR} (see Theorem 4 in Appendix A). That is, we represent algorithm~\eqref{eq:alg2} in the form suitable for stochastic approximation analysis.
	  We start by substituting regression model~\eqref{eq:obs_vec} into algorithm~\eqref{eq:algcomp}, we get 
	\begin{align}\label{eq:algonoise}
		\mathbf{x}^{t+1}=\mathbf{x}^{t}-\alpha_t\left( \frac{b}{a}\mathbf{L}_{\boldsymbol{\Psi}_\mathrm{c}}(\mathbf{x})-\mathbf{H}^{\top}\boldsymbol{\Psi}_\mathrm{o}\left(     
		\mathbf{H}\left(\mathbf{1}_{N}\otimes \boldsymbol{\theta}^{\ast}\right)+\mathbf{n}^{t}
		-\mathbf{H}\mathbf{x}^{t}\right)\right).
	\end{align}
	Define $\boldsymbol{\zeta}^t\in\mathbb{R}^{N}$ and  $\boldsymbol{\eta}^t\in\mathbb{R}^{MN}$ by 
	\begin{align}\label{eq:zeta}
		\boldsymbol{\zeta}^t=\boldsymbol{\Psi}_\mathrm{o}(\mathbf{H}\left(\mathbf{1}_{N}\otimes \boldsymbol{\theta}^{\ast}\right)+\mathbf{n}^{t}
		-\mathbf{H}\mathbf{x}^{t})-\boldsymbol{\varphi}_\mathrm{o}\left(\mathbf{H}\left(\left(\mathbf{1}_{N}\otimes \boldsymbol{\theta}^{\ast}\right)
		-\mathbf{x}^{t}\right)\right), \quad
		\boldsymbol{\eta}^t=\begin{bmatrix}
			\vdots\\
			\sum\limits_{j\in\Omega_{i}}\boldsymbol{\eta}_{ij}^t\\
			\vdots
		\end{bmatrix},
	\end{align}
	where $\boldsymbol{\eta}_{ij}^t=\boldsymbol{\Psi}_\mathrm{c}(\mathbf{x}_i^t-\mathbf{x}_j^t+\boldsymbol{\xi}_{ij}^t)-\boldsymbol{\varphi}_\mathrm{c}(\mathbf{x}_i^t-\mathbf{x}_j^t).$ Now, {since $\varphi$ is defined by~\eqref{eq:phi-def}}, it can be shown that  $\mathbb{E}[\boldsymbol{\zeta}^t]=\mathbb{E}[\boldsymbol{\eta}^t]=0$, where the expectation is taken with respect to $\mathcal{F}$ (see Appendix B).
	Furthermore, we define function $\mathbf{L}_{\boldsymbol{\varphi}_\mathrm{c}}:\mathbb{R}^{MN}\to\mathbb{R}^{MN}$ as $\mathbf{L}_{\boldsymbol{\varphi}_\mathrm{c}}(\cdot)=\mathbf{L}_{\boldsymbol{\Psi}_\mathrm{c}}(\cdot)-\boldsymbol{\eta}^t,$ i.e., its $i$-th block of size $M$ is $\sum\limits_{j\in\Omega_{i}}\boldsymbol{\varphi}_\mathrm{c}(\mathbf{x}_i-\mathbf{x}_j).$ for $i=1,2,...,N.$
	Finally, substituting~\eqref{eq:zeta} into~\eqref{eq:algonoise}, we rewrite algorithm~\eqref{eq:algcomp} by
	\begin{align}\label{eq:algfin}
		\mathbf{x}^{t+1}=\mathbf{x}^{t}-\alpha_t\left( \frac{b}{a}\mathbf{L}_{\boldsymbol{\varphi}_\mathrm{c}}(\mathbf{x}^t)- \mathbf{H}^{\top}\boldsymbol{\varphi}_\mathrm{o}\left(     
		\mathbf{H}\left(\left(\mathbf{1}_{N}\otimes \boldsymbol{\theta}^{\ast}\right)
		-\mathbf{x}^{t}\right)\right) - \mathbf{H}^{\top}\boldsymbol{\zeta}^t+\frac{b}{a}\boldsymbol{\eta}^t \right).
	\end{align}
 Now, we are ready to establish following results.
 
	\subsection{Almost sure convergence}\label{subsection:as}
	We have the following Theorem.
	\begin{theorem}[Almost sure convergence]
		\label{theorem-almost-surely}
		Let Assumptions \ref{as:newtworkmodelandobservability}-\ref{as:nonlinearity} hold and $\alpha_t= a/(t+1)^\delta,$ $\delta\in (0.5,1]$.  
		Then, for each agent $i=1,...,N$, 
		the sequence of iterates $\{\mathbf{x}_i^t\}$ 
		generated by algorithm~\eqref{eq:alg2} 
		converges almost surely to the true 
		vector parameter~$\boldsymbol{\theta}^{\ast}$.
	\end{theorem} 
	\noindent Theorem~\ref{theorem-almost-surely} establishes almost sure convergence of the proposed algorithm~\eqref{eq:alg2}, whether observation or communication noises have finite or infinite moments of order greater then one. 
 On the other hand, if we set at least one of the functions ${\Psi}_\mathrm{o}, {\Psi}_\mathrm{c}$ to be identity functions (and thus recover either the $\mathcal{LU}$ scheme from~\cite{KMR} or the method from~\cite{Ourwork}), the resulting method fails to converge (See Appendix~D). In other words, the methods in~\cite{KMR} and~\cite{Ourwork} fail to converge under the simultaneous presence of heavy-tailed observation and communication noises.
	
	\begin{proof} (Proof of Theorem~\ref{theorem-almost-surely})
		\\\noindent The proof consists of verifying conditions B1--B5 of Theorem~4 (See Appendix A).
        {First, we define quantities $\mathbf{r}(\mathbf{x})$ and $\boldsymbol{\gamma}(t+1,\mathbf{x},\omega)$ by:
	\begin{align}
		\mathbf{r}(\mathbf{x})&= -\frac{b}{a}\mathbf{L}_{\boldsymbol{\varphi}_\mathrm{c}}(\mathbf{x})- \mathbf{H}^{\top}\boldsymbol{\varphi}_\mathrm{o}\left(     
		\mathbf{H}\left(\mathbf{x}-\left(\mathbf{1}_{N}\otimes \boldsymbol{\theta}^{\ast}\right)
		\right)\right),\label{eqn:r}\\
		\boldsymbol{\gamma}(t+1,\mathbf{x},\omega)&=-\frac{b}{a}\boldsymbol{\eta}^t+\mathbf{H}^{\top}\boldsymbol{\zeta}^t. \label{eqn:gamma}
	\end{align}
	Here, $\omega$ denotes a canonical element of the underlying probability space $(\Omega, \mathcal{F}, \mathbb {P})$.}
		
		\noindent Condition B1 holds because $\mathbf{r}(\cdot)$ is $\mathcal{B}^{MN}$ measurable and $\boldsymbol{\gamma}(t + 1, \cdot, \cdot)$ is $\mathcal{B}^{MN} \otimes \mathcal{F}$ measurable for each $t$, where $\mathcal{B}^{MN}$ is the Borel sigma algebra on $\mathbb{R}^{MN}$. Consider the filtration $\mathcal{F}_t$, $t=1,2,...,$ where $\mathcal{F}_t$ is the $\sigma$- algebra generated by $\{\mathbf{n}^s\}_{s=0}^{t-1}$ and $\{\boldsymbol{\xi}^s_{ij}\}_{s=0}^{t-1}.$ We have that the family of random vectors 
		$\boldsymbol{\gamma}(t+1,\mathbf{x},\omega)$ is $\mathcal{F}_t$ measurable, zero-mean and independent of $\mathcal{F}_{t-1}$. Hence, condition B2 holds.
		
		\noindent We now show that condition B3 also holds. We use the following Lyapunov function $V:\mathbb{R}^{MN}\to\mathbb{R},$ \begin{align}\label{eqn:lyapunov}
		    V(\mathbf{x})=||\mathbf{x}-\mathbf{1}_N\otimes\boldsymbol{\theta}^\ast||^2,
		\end{align}
		which is clearly twice continuously differentiable and has uniformly bounded second order partial derivatives. The gradient of $V$ equals $\nabla V(\mathbf{x}) = 2\left(\mathbf{x}-\mathbf{1}_{N}\otimes \boldsymbol{\theta}^{\ast}\right)$. We must show that
		\begin{align}\label{eq:innerproduct}
		\sup\limits_{\mathbf{x}\in S_\epsilon}\langle\mathbf{r}(\mathbf{x}),\nabla V(\mathbf{x})\rangle<0,
		\end{align}
		where $S_\epsilon=\{
		\mathbf{x} \in {\mathbb R}^{MN}:\,
		\|\mathbf{x} - \mathbf{1}_{N}\otimes \boldsymbol{\theta}^{\ast}\| \in (\epsilon,1/\epsilon) \}$. 
		For any $\mathbf{x}\in\mathbb{R}^{MN}$, we have:
		\begin{align}\label{eqn:T1T2}
			&\langle \mathbf{r}(\mathbf{x}), \nabla V(\mathbf{x})\rangle = 2\left(\mathbf{x}-\mathbf{1}_{N}\otimes \boldsymbol{\theta}^{\ast}\right)^{\top}\left(-\frac{b}{a}\mathbf{L}_{\boldsymbol{\varphi}_\mathrm{c}}(\mathbf{x})- \mathbf{H}^{\top}\boldsymbol{\varphi}_\mathrm{o}\left(     
			\mathbf{H}\left(\mathbf{x}^{t}-\left(\mathbf{1}_{N}\otimes \boldsymbol{\theta}^{\ast}\right)
			\right)\right)\right)\nonumber\\
			&=-\frac{2b}{a}\underbrace{\left(\mathbf{x}-\mathbf{1}_{N}\otimes \boldsymbol{\theta}^{\ast}\right)^{\top}\mathbf{L}_{\boldsymbol{\varphi}_\mathrm{c}}(\mathbf{x})}_{\text{$T_1(\mathbf{x})$}}
			-\underbrace{\left(     
				\mathbf{H}\left(\mathbf{x}-\left(\mathbf{1}_{N}\otimes \boldsymbol{\theta}^{\ast}\right)
				\right)\right)^{\top}\boldsymbol{\varphi}_\mathrm{o}\left(     
				\mathbf{H}\left(\mathbf{x}-\left(\mathbf{1}_{N}\otimes \boldsymbol{\theta}^{\ast}\right)
				\right)\right)}_{\text{$T_2(\mathbf{x})$}}.
		\end{align}
		The terms $T_1(\mathbf{x})$ and $T_2(\mathbf{x})$ can be written respectively as
		\begin{align*}
			T_1(\mathbf{x}) &= \sum_{\{i,j\}\in E,\, i<j}\left(\mathbf{x}_{i}-\mathbf{x}_{j}\right)^{\top}\boldsymbol{\varphi}_\mathrm{c}\left(\mathbf{x}_{i}-\mathbf{x}_{j}\right)=\sum_{\{i,j\}\in E,\, i<j}\mathbf{g}{\top}\boldsymbol{\varphi}_\mathrm{c}\left(\mathbf{g}\right), \\ T_2(\mathbf{x})&=\sum_{i=1}^{N} \hat{\mathbf{g}}_i \,\varphi_\mathrm{o} (\hat{\mathbf{g}}_i),
		\end{align*}
		where $\hat{\mathbf{g}}=\mathbf{H}^{\top}\boldsymbol{\varphi}_\mathrm{o}\left(     
		\mathbf{H}\left(\mathbf{x}^{t}-\left(\mathbf{1}_{N}\otimes \boldsymbol{\theta}^{\ast}\right)
		\right)\right)$, $\mathbf{g}=\mathbf{x}_{i}-\mathbf{x}_{j}$ and $\mathbf{g}{\top}\boldsymbol{\varphi}_\mathrm{c}\left(\mathbf{g}\right)=\sum_{\ell=1}^{M}\mathbf{g}_{\ell} {\varphi_\mathrm{c}}\left(\mathbf{g}_{\ell}\right)$. Using the fact that both of the functions $\varphi_\mathrm{c}$ and $\varphi_\mathrm{o}$ are odd functions, for which we have that $\varphi(a)>0$ if $a>0$, we have that $\langle \mathbf{r}(\mathbf{x}), \nabla V(\mathbf{x})\rangle\geq0$ for all $\mathbf{x}\in\mathbb{R}^{MN}$ (see Appendix B). Moreover, recalling the fact that function $\varphi_\mathrm{c}$ is continuous at zero, and equal to zero only at zero, we have that $T_1(\mathbf{x})$ is equal to zero if and only if $\mathbf{x}-\mathbf{1}_{N}\otimes \boldsymbol{\theta}^{\ast}=\mathbf{1}_{N}\otimes \mathbf{m},$ for $\mathbf{m}\in\mathbb{R}^M$ (see Lemma~6 in Appendix B). We only consider the case when $\mathbf{m}\neq0$, since from $\mathbf{m}=0$  we have that $\mathbf{x}=\mathbf{1}_{N}\otimes \boldsymbol{\theta}^{\ast}$, which is not in the set $S_\epsilon.$
		However, for that choice of $\mathbf{x}-\mathbf{1}_{N}\otimes \boldsymbol{\theta}^{\ast}$ we have that
		\begin{align*}
			T_2(\mathbf{1}_{N}\otimes \boldsymbol{\theta}^{\ast}+\mathbf{1}_{N}\otimes \mathbf{m})&=\left(     
			\mathbf{H} \, \mathbf{1}_{N}\otimes \mathbf{m}\right)^{\top}\boldsymbol{\varphi}_\mathrm{o}\left(     
			\mathbf{H}\,\mathbf{1}_{N}\otimes \mathbf{m} \right)\\
			&=\sum_{i=1}^N \left( \mathbf{h}_i^\top \mathbf{m}\right) \varphi_\mathrm{o} \left(\mathbf{h}_i^\top \mathbf{m}\right)>0,
		\end{align*}
		since $\mathbf{h}_i^\top \mathbf{m}$ and $\varphi_\mathrm{o} \left(\mathbf{h}_i^\top \mathbf{m}\right)$ have the same sign. Hence, for all $\epsilon>0$ we have that $\sup\limits_{\mathbf{x}\in S_\epsilon}\langle\mathbf{r}(\mathbf{x}),\nabla V(\mathbf{x})\rangle<0.$ Thus, condition B3 also holds.
		
		\noindent Now we inspect condition B4. From equation~\eqref{eqn:r} we have that 
		\begin{align}\label{eq:boundedr}
			\left\|\mathbf{r}(\mathbf{x})\right\|^2 \leq  \left\| \frac{b}{a}\mathbf{L}_{\boldsymbol{\varphi}_\mathrm{c}}(\mathbf{x}-\mathbf{1}_{N}\otimes \boldsymbol{\theta}^{\ast})\right\|^2 +\left\| \mathbf{H}^{\top}\boldsymbol{\varphi}_\mathrm{o}\left(     
			\mathbf{H}\left(\mathbf{x}-\left(\mathbf{1}_{N}\otimes \boldsymbol{\theta}^{\ast}\right)
			\right)\right)\right\|^2
			\leq c_1 (1+V(\mathbf{x})),
		\end{align}
		for some positive constant $c_1$ (see Appendix B).
		Moreover, we have that 
		\begin{align}\label{eq:gammabound}
			\left\|\boldsymbol{\gamma}(t+1,\mathbf{x},\omega)\right\|^2&\leq\left\|\frac{b}{a}\boldsymbol{\eta}^t\right\|^2+\left\|\mathbf{H}^{\top}\boldsymbol{\zeta}^t\right\|^2
		\end{align}
		which leads to
		\begin{align}\label{eq:boundedg}
			\mathbb{E}\left[\left\|\boldsymbol{\gamma}(t+1, \mathbf{x}^{t}, \omega)\right\|^2\right] \leq c_2(1+V(\mathbf{x})), 
		\end{align}
		for some positive constant $c_2.$
		Finally, we have that 
		\begin{align*}
			\left\|\mathbf{r}(\mathbf{x})\right\|^2 +\mathbb{E}\left[\left\|\boldsymbol{\gamma}(t+1, \mathbf{x}^{t}, \omega)\right\|^2\right]\leq c_3(1+V(\mathbf{x})),
		\end{align*}
		for some positive constant $c_3.$ Setting that $\epsilon\to0^+$ in~\eqref{eq:innerproduct}, for all $\mathbf{x}\in\mathbb{R}^{MN},$ we have that $\langle\mathbf{r}(\mathbf{x}),\nabla V(\mathbf{x}) \rangle \leq 0.$  Thus,
		\begin{align*}
		\left\|\mathbf{r}(\mathbf{x})\right\|^2 +\mathbb{E}\left[\left\|\boldsymbol{\gamma}(t+1, \mathbf{x}^{t}, \omega)\right\|^2\right]\leq c_3(1+V(\mathbf{x}))-k\langle\mathbf{r}(\mathbf{x}),\nabla V(\mathbf{x})\rangle
		\end{align*}
		for every $k>0.$ Therefore, condition B4 also holds. Condition B5 holds by the definition of the algorithm~\eqref{eq:alg2}. Thus, almost sure convergence is proved.
	\end{proof}
	\subsection{Asymptotic normality}\label{subsection:an}
	We now consider asymptotic normality of the proposed estimator~\eqref{eq:alg2}. We have the following theorem.
	\begin{theorem}[Asymptotic normality]
		\label{theorem-asymptotic-normality}
		Let Assumptions \ref{as:newtworkmodelandobservability}-\ref{as:nonlinearity} hold. 
		Consider algorithm~\eqref{eq:alg2} with step-size 
		$\alpha_t=a/(t+1)^{\delta}$, $t=0,1,...,$ $a>0$, with  $\delta=1$. Then, 
		the normalized sequence of iterates $\{
		\sqrt{t+1} (  \mathbf{x}^t-\mathbf{1}_N\otimes\boldsymbol{\theta}^\ast ) \}$ 
		converges in distribution to a zero-mean multivariate normal random vector, i.e., the following holds:
		\begin{align*}
			\sqrt{t+1}(\mathbf{x}^t-\mathbf{1}_N\otimes\boldsymbol{\theta}^\ast)\Rightarrow\mathcal{N}(\mathbf{0},\mathbf{S}),
		\end{align*}
		where the asymptotic covariance matrix $\mathbf{S}$ equals:
		\begin{equation}
			\label{eqn-asympt-var}
			\mathbf{S}=a^2\int\limits_{0}^\infty e^{\boldsymbol{\Sigma} v}\mathbf{S}_0e^{\boldsymbol{\Sigma}^\top v}dv.
		\end{equation}
		Here,  
		$\mathbf{S}_0=\frac{b^2}{a^2}\sigma_\mathrm{c}^2\Diag\left( \{d_i\,\mathbf{I}_M\}\right)-\frac{b}{a}\mathbf{K}_\mathrm{c,o} \mathbf{H} -\frac{b}{a}\mathbf{H}^\top\mathbf{K}_\mathrm{c,o}^\top  +\sigma_\mathrm{o}^2\mathbf{H}^\top\mathbf{H}$;
		$\sigma_\mathrm{o}^2 = \int |\Psi_\mathrm{o}(w)|^2 $ $d \Phi_\mathrm{o}(w)$
		is the effective observation noise variance 
		after passing through the nonlinearity~$\Psi_\mathrm{o}$;
		$\sigma_\mathrm{c}^2 = \int |\Psi_\mathrm{c}(w)|^2 d \Phi_\mathrm{c}(w)$
		is the effective communication noise variance 
		after passing through the nonlinearity~$\Psi_\mathrm{c}$; $\mathbf{K}_\mathrm{c,o}\in\mathbf{R}^{MN\times N}$ is the effective cross-covariance matrix between the observation and the communication noise after passing through the appropriate nonlinearity, i.e., the $(k,s)$ element of the matrix $\mathbf{K}_\mathrm{c,o}$ is given by {$[(\mathbf{K}_\mathrm{c,o})]_{ks}=\sum\limits_{j\in\Omega_i}\int\int\Psi_\mathrm{c}(w_{ij\ell})\Psi_\mathrm{o}(w_k)p^{\mathrm{c,o}}_{k,ij\ell}(w_{ij\ell},w_k)dw_{ij\ell}dw_k.$ Here, $\ell$ satisfies the following: $s=M(i-1)+\ell$; and $p^{\mathrm{c,o}}_{k,ij\ell}$ is the joint probability density function} for the $k$-th observation noise $n_k$ and the  $\ell$-th element of the communication noise $[(\boldsymbol{\xi}_{ij})]_\ell$.
		We also recall the observation matrix $\mathbf{H}$ in \eqref{eq:obs_vec};  
		functions $\varphi_\mathrm{c}$, $\varphi_\mathrm{o}$ appropriate versions of function $\varphi$ in~\eqref{eq:phi-def};  
		and $\Sigma=\frac{1}{2}\mathbf{I}-a(\frac{b}{a}\varphi_\mathrm{c}^\prime(0)\mathbf{L}\otimes \mathbf{I}_M + \varphi_\mathrm{o}^\prime(0) \mathbf{H}^{\top}\mathbf{H});$ here, $a$ is taken large enough such that matrix $\boldsymbol{\Sigma}$ is stable.
	\end{theorem}
	
	\begin{remark}
Notice that, for the assumed setting,
$\sigma_{\mathrm{c}}^2$ and $\sigma_{\mathrm{o}}^2$
 are finite. Also, $\mathbf{K}_\mathrm{c,o}$ is finite, i.e.,  $\|\mathbf{K}_\mathrm{c,o}\|<\infty$,  since we have that
    \begin{align*}
        |\int\int\Psi_\mathrm{c}(w_1)\Psi_\mathrm{o}(w_2)d\Phi^{\mathrm{c,o}}|\leq \int\int|\Psi_\mathrm{c}(w_1)\Psi_\mathrm{o}(w_2)|d\Phi^{\mathrm{c,o}}<\frac{1}{2}\sigma_{\mathrm{c}}^2+\frac{1}{2}\sigma_{\mathrm{o}}^2.
    \end{align*}
    %
	\end{remark}

	\begin{remark}
	If we assume that observation and communication noise are mutually independent, the only difference from the previous theoretical results occurs in the $\mathbf{A}(t,\mathbf{x})$, i.e., in the $\mathbf{S}_0$. Under this setting, matrix $\mathbf{S}_0$ is now equal to
	\begin{align*}
		\mathbf{S}_0=\frac{b^2}{a^2}\sigma_\mathrm{c}^2\Diag\left( \{d_i\,\mathbf{I}_M\}\right)+\sigma_\mathrm{o}^2\mathbf{H}^\top\mathbf{H},
	\end{align*}
	which is expected, since the effective cross-covariance matrix $\mathbf{K}_\mathrm{c,o} $ is now equal to zero.
	\end{remark}
	
	Theorem~\ref{theorem-asymptotic-normality} establishes asymptotic normality of the proposed method. This is achieved with heavy-tailed observation and communication noise an the nonlinearities $\Psi_\mathrm{o}$ and $\Psi_\mathrm{c}$ with uniformly bounded outputs. Moreover, the theorem explicitly evaluates the corresponding asymptotic variance. When the two noises are mutually independent, $\Psi_\mathrm{o}$ is identity, and observation noise variance is finite, we recover the result in~\cite{Ourwork}, Theorem 3.5, as a special case. That is, a notable difference with respect to~\cite{Ourwork} is the ability to handle here mutually correlated observation and communication noises. The effect of correlation is complex in general, however, as shown in Section~\ref{section:examples} later, generally a stronger positive noises correlation leads to a lower asymptotic variance. Intuitively, at an extreme, a full positive correlation  practically means that only one effective noise exists in the system, and hence it can be suppressed more easily. Further, note that Theorem~\ref{theorem-asymptotic-normality} establishes a local asymptotic rate $O(1/t)$ of $\mathbf{x}^t$ to zero, in the weak convergence sense, when $\alpha_t=a/(t+1).$ We show later (see Theorem~\ref{theorem-MSE}) that a global MSE rate $O(1/t^{\hat{\delta}})$ with a lower (worse) degree $\hat{\delta}$ can be established when 
	step-size $\alpha_t=a/(t+1)^{\delta}$, 
	$\delta \in (0.5,1)$, is used.

    {We next discuss asymptotic efficiency\footnote{An estimator $\mathbf{y}^t$ of an unknown parameter $\boldsymbol{\theta}^\star$, for which we have that $\sqrt{t+1}(\mathbf{y}^t-\boldsymbol{\theta}^\ast)\Rightarrow\mathcal{N}(\mathbf{0},\boldsymbol{\Sigma})$, is said to be asymptotically efficient if $\mathbf{S}=\mathbf{I}^{-1}(\boldsymbol{\theta}^\star),$ where $\mathbf{I}(\boldsymbol{\theta}^\star)$ is the Fisher information matrix. The Fisher information matrix represents the best achievable asymptotic covariance by any estimator, as determined by the well-known Cramer-Rao bound (see \cite{Nevelson}).} of the proposed estimator. We first briefly review the relevant existing work to better position our results.  
    First, consider the best linear centralized estimator $\mathbf{x}_{\mathrm{cent}}^t$ of $\boldsymbol{\theta}^\star,$ that has access to measurements from all sensors (nodes) $n=1,2,...,N$ at all times $t=0,1,...$. In the general case, additionally assuming that observation noise has finite variance, the best linear centralized estimator $\mathbf{x}_{\mathrm{cent}}^t$ is asymptotically normal and has the lowest asymptotic covariance matrix $\mathbf{S}_{\mathrm{cent}}$ among all estimators of $\boldsymbol{\theta}^\star$ when the only knowledge of observation noise is variance and no other information of noise distribution is known.
Moreover, its asymptotic covariance matrix $\mathbf{S}_{\mathrm{cent}}$ attains the Cramér-Rao lower bound if the observation noise is Gaussian (see for example~\cite{SoummyaAdaptive}). 
    On the other hand, when the probability density function is known, the centralized estimator in~\cite{PolyakNL} can be tuned to the pdf of the observation noise so that it achieves the Cramér-Rao bound. In the distributed setting, when there is no communication noise, the authors of~\cite{SoummyaAdaptive} develop an estimator which is asymptotically normal and has the optimal asymptotic covariance matrix $\mathbf{S}_{\mathrm{cent}}$ (optimal in the sense that the asymptotic covariance matrix is the same as for the best linear centralized estimator $\mathbf{x}_{\mathrm{cent}}^t$). We now discuss the asymptotic covariance matrix $\mathbf{S}$ of the proposed estimator~\eqref{eq:alg2}. This quantity depends on the system parameters, including network topology and communication noise. Therefore, in the general case, the proposed estimator~\eqref{eq:alg2} is not asymptotically efficient, i.e., $\mathbf{S}\neq \mathbf{I}^{-1}(\boldsymbol{\theta}^\star),$
 where $\mathbf{I}(\boldsymbol{\theta}^\star)$ is the Fisher information matrix. However, with respect to the proposed distributed recursive estimator, we make the following observations. 1) First, 
the estimator is order-optimal in the weak convergence sense; that is,  its (weak convergence sense) rate of error decay is the same as that of the asymptotically efficient estimator. 2) The corresponding ``convergence constant,'' i.e., the asymptotic covariance, is different from that of the centralized Cramér-Rao-optimal estimator, and it is hence not optimal. We note that the paper provides major contributions with respect to state of the art, as it gives the first distributed estimator that ensures almost sure convergence in the presence of infinite variance correlated sensing and communication noises; moreover, its weak convergence sense rate of convergence is order-optimal. It remains an interesting future work direction to explore whether an 
optimal asymptotic covariance can be achieved in this setting via distributed estimators. In view of the results~\cite{PolyakNL} for the centralized setting, it is likely that this cannot be achieved unless the nonlinearities are tuned to the noise pdfs that in turn have to be known.}
	
	\begin{proof} (Proof of Theorem~\ref{theorem-asymptotic-normality}) 
		\\\noindent We prove Theorem~\ref{theorem-asymptotic-normality} in the same manner as Theorem~\ref{theorem-almost-surely} is proved, i.e., by verifying assumptions C1-C5 of Theorem~4 in (see Appendix A).
		Function $\mathbf{r}(\cdot)$ defined by~\eqref{eqn:r} can be written as
		\begin{align*}
			\mathbf{r}(\mathbf{x})= -\frac{b}{a}\varphi_\mathrm{c}^\prime(0) \mathbf{L}\otimes \mathbf{I}_M\,\left(\mathbf{x} - \mathbf{1}_{N}\otimes \boldsymbol{\theta}^{\ast} \right)- \varphi_\mathrm{o}^\prime(0) \mathbf{H}^{\top}     
			\mathbf{H}\left(\mathbf{x}-\mathbf{1}_{N}\otimes \boldsymbol{\theta}^{\ast}
			\right)+\boldsymbol{\delta}(\mathbf{x}),
		\end{align*}
		Here, mapping $\boldsymbol{\delta}:\mathbb{R}^{MN}\to\mathbb{R}^{MN}$ is given by: 
		\begin{align}\label{eq:deltabold}
			\boldsymbol{\delta}(\mathbf{x})=-\frac{b}{a}\mathbf{L}_{\boldsymbol{\delta}_\mathrm{c}}(\mathbf{x})-\mathbf{H}^\top\boldsymbol{\delta}_\mathrm{o}\left( \mathbf{H}\left(\mathbf{x}-\mathbf{1}_{N}\otimes \boldsymbol{\theta}^{\ast}
			\right)\right).
		\end{align}
		Next, mapping $\mathbf{L}_{\boldsymbol{\delta}_\mathrm{c}}(\mathbf{x}):\mathbb{R}^{MN}\to\mathbb{R}^{MN}$ is vector of size $MN$ such that the $i$-th $M$-size block equals $\sum\limits_{j\in\Omega_{i}}\boldsymbol{\delta}_\mathrm{c}(\mathbf{x}_i-\mathbf{x}_j),$ $i=1,2,...,N,$
		mappings $\boldsymbol{\delta}_\mathrm{c}:\mathbb{R}^M\to\mathbb{R}^M$,  $\boldsymbol{\delta}_\mathrm{o}:\mathbb{R}^N\to\mathbb{R}^N$ are component-wise maps of $\delta_\mathrm{c}$ and $\delta_\mathrm{o}$ 
        are first order residuals that corresponds to  $\varphi_\mathrm{c}$ and $\varphi_\mathrm{o}$ respectively, i.e., $\boldsymbol{\delta}_\mathrm{c}(\mathbf{y}_1,\mathbf{y}_1,...,\mathbf{y}_M)=[\delta_\mathrm{c}(\mathbf{y}_1),\delta_\mathrm{c}(\mathbf{y}_2),...,\delta_\mathrm{c}(\mathbf{y}_M)]^\top$ and $\boldsymbol{\delta}_\mathrm{o}(\hat{\mathbf{y}}_1,\hat{\mathbf{y}}_1,...,\hat{\mathbf{y}}_N)=[\delta_\mathrm{o}(\hat{\mathbf{y}}_1),\delta_\mathrm{o}(\hat{\mathbf{y}}_2),...,\delta_\mathrm{o}(\hat{\mathbf{y}}_M)]^\top$ for $\mathbf{y}\in\mathbb{R}^N, \hat{\mathbf{y}}\in\mathbb{R}^M$ (see Appendix B). 
		
		\noindent Thus, $\mathbf{r}(\mathbf{x})$ admits representation in~(36) of Theorem~4 in Appendix A for $\mathbf{B}=-\frac{b}{a}\varphi_\mathrm{c}^\prime(0)\mathbf{L}\otimes \mathbf{I}_M - \varphi_\mathrm{o}^\prime(0) \mathbf{H}^{\top}\mathbf{H}$ and mapping $\boldsymbol{\delta}(\cdot)$ defined by~\eqref{eq:deltabold}.	Therefore, condition C1 holds. Since we use that $\alpha_{t}=\frac{a}{t+1}$, condition C2 trivially holds. Furthermore, $\boldsymbol{\Sigma}=a\mathbf{B}+\frac{1}{2}\mathbf{I}$ is stable if $a$ is large enough, because matrix $-\mathbf{B}$ is positive definite (See \cite{KMR}). Thus, condition C3 also holds.
		
		\noindent For $\mathbf{A}(t,\mathbf{x})=\mathbb{E}\big[\boldsymbol{\gamma}(t+1,\mathbf{x},\omega)\boldsymbol{\gamma}^{\top}(t+1,\mathbf{x},\omega)\big]$ it is easy to show that
		\begin{align*}        \lim\limits_{t\to\infty,\mathbf{x}\to\boldsymbol{\theta}^\ast}\mathbf{A}(t,\mathbf{x})=\frac{b^2}{a^2}\sigma_\mathrm{c}^2\Diag\left( \{d_i\,\mathbf{I}_M\}\right) - \frac{b}{a}\mathbf{K}_\mathrm{c,o} \mathbf{H} -\frac{b}{a}\mathbf{H}^\top\mathbf{K}_\mathrm{c,o}^\top   +\sigma_\mathrm{o}^2\mathbf{H}^\top\mathbf{H}.
		\end{align*}
 		Therefore, condition C4 also holds. To show that condition~C5 holds, it is suffice to show that the family of random variables $\{\|\boldsymbol{\gamma_{\boldsymbol{\varphi}}}(t+1,\mathbf{x},\omega)\|^2\}_{t=0,1,..., \,\|\mathbf{x}-\boldsymbol{\theta}^\star\|< \epsilon}$ is uniformly integrable. To do that, follow the arguments as in e.g., \cite{KMR} and~\cite{Ourwork}. \end{proof}

	\subsection{Mean squared error convergence}\label{subsection-MSE}
	In this subsection, we state and prove a result on the mean squared error (MSE) convergence rate when both nonlinearities $\Psi_{\mathrm{o}}$ and $\Psi_{\mathrm{c}}$ satisfy part 5' of Assumption~\ref{as:nonlinearity}, i.e., $|\Psi_{\mathrm{o}}|\leq c_{\mathrm{o}},$ $|\Psi_{\mathrm{c}}|\leq c_{\mathrm{c}},$ for some positive constants $c_{\mathrm{o}}$ and $c_{\mathrm{c}}$. Moreover, we set the step size to $\alpha_t=\frac{a}{(t+1)^\delta},$ for $\delta\in(\frac{1}{2},1).$ We have the following theorem.

    \begin{theorem}[MSE convergence]\label{theorem-MSE} Let Assumptions \ref{as:newtworkmodelandobservability}-\ref{as:nonlinearity} hold. Then, for the sequence of iterates $\{\mathbf{x}^t\}$ 
		generated by algorithm~\eqref{eq:algcomp}, provided that the step-size sequence $\{\alpha_t\}$ is given by $\alpha_t=a/(t+1)^\delta,$ $a>0,\delta\in(0.5,1)$, there exists $\hat{\delta} \in (0,1)$ such that $\mathbb{E}[\|\mathbf{x}-\mathbf{1}_{N}\otimes \boldsymbol{\theta}^{\ast}\|^2]=O(1/t^{\hat{\delta}}).$ 
	\end{theorem}
	Theorem~\ref{theorem-MSE} establishes 
	a MSE convergence rate of the 
	proposed estimator~\eqref{eq:algcomp} under the simultaneous 
	presence of heavy-tailed (possibly infinite variance) observation and communication noises, when both the observation and communication nonlinearities have uniformly bounded outputs.
	 This is in contrast with recent studies on 
	 distributed estimation in heavy-tailed noises like~\cite{Ourwork}
	 that only establishes a.s. and asymptotic normality results. 
	 We refer to the proof of Theorem~\ref{theorem-MSE} 
	 for the exact value of the convergence rate power~$\hat{\delta}$.

	\textbf{Setting up the proof.} 
	  We now prove Theorem~\ref{theorem-MSE} through a 
	  sequence of intermediate results (Lemmas).
	  Recall quantities $\mathbf{r}(\cdot),$ $\boldsymbol{\gamma}(\cdot,\cdot,\cdot)$ and $V(\cdot)$ from \eqref{eqn:r}, $\eqref{eqn:gamma}$ and \eqref{eqn:lyapunov} respectively. 
	   The proof will be based on establishing a sufficient decay 
	   on quantity $\mathbb{E}[V(\mathbf{x}^t)]$. First, notice that  algorithm~\eqref{eq:algfin} can be written as
	\begin{align*}
	    \mathbf{x}^{t+1}=\mathbf{x}^{t}+\alpha_t\left(\mathbf{r}(\mathbf{x}^t)+\boldsymbol{\gamma}(t+1,\mathbf{x}^t,\omega)\right).
	\end{align*}
	Moreover, we have that
	\begin{align*}
	    V(\mathbf{x}^{t+1})&=V(\mathbf{x}^t)+2\alpha_t\left(\mathbf{x}^t-\mathbf{1}_{N}\otimes \boldsymbol{\theta}^{\ast}\right)^\top\left(\mathbf{r}(\mathbf{x}^t)+\boldsymbol{\gamma}(t+1,\mathbf{x}^t,\omega)\right)\\&+\alpha_t^2\|\mathbf{r}(\mathbf{x}^t)+\boldsymbol{\gamma}(t+1,\mathbf{x}^t,\omega)\|^2\\
	    &= V(\mathbf{x}^t)+2\alpha_t\left(\mathbf{x}^t-\mathbf{1}_{N}\otimes \boldsymbol{\theta}^{\ast}\right)^\top\left(\mathbf{r}(\mathbf{x}^t)+\boldsymbol{\gamma}(t+1,\mathbf{x}^t,\omega)\right)+\alpha_t^2\,c',
	\end{align*}
	for positive constant $c'=\|\mathbf{r}(\mathbf{x}^t)+\boldsymbol{\gamma}(t+1,\mathbf{x}^t,\omega)\|^2<\infty.$ Therefore, taking a conditional expectation 
	with respect to $\mathcal{F}_t$, we have:
    \begin{align}\label{eq:expfilt}
        \mathbb{E}[V(\mathbf{x}^{t+1})|\mathcal{F}_t]=V(\mathbf{x}^t)+2\alpha_t\left(\mathbf{x}^t-\mathbf{1}_{N}\otimes \boldsymbol{\theta}^{\ast}\right)^\top\mathbf{r}(\mathbf{x}^t)+\alpha_t^2\,c'.
    \end{align}
    Also, from equation~\eqref{eqn:T1T2}, it follows that 
    \begin{align}\label{eqn:boundT1T2}
        \left(\mathbf{x}^t-\mathbf{1}_{N}\otimes \boldsymbol{\theta}^{\ast}\right)^\top\mathbf{r}(\mathbf{x}^t)=-\frac{b}{a}T_1(\mathbf{x}^t)-T_2(\mathbf{x}^t).
    \end{align}
    We next need to show that the quantity in \eqref{eqn:boundT1T2}
     is ``sufficiently negative'', relative to 
     quantity $V(\mathbf{x}^t)$. 
     This is achieved through a sequence of lemmas. First, we upper bound  quantities $\|\mathbf{x}^t\|$ and $\|\mathbf{x}^t-\mathbf{1}_{N}\otimes \boldsymbol{\theta}^{\ast}\|.$
	\begin{lemma}\label{lemma:boundsonx}
	    Let Assumptions \ref{as:newtworkmodelandobservability}-\ref{as:nonlinearity} hold. Then, for the sequence of iterates $\{\mathbf{x}^t\}$ 
		generated by algorithm~\eqref{eq:algcomp}, provided that the step-size sequence $\{\alpha_t\}$ is given by $\alpha_t=a/(t+1)^\delta,$ $a>0,\delta\in(0.5,1),$ we have that, for any outcome $\omega$:
		\begin{align}
		    \|\mathbf{x}^t\|&\leq g_t=\|\mathbf{x}^0\|+\left(b\,\sqrt{MN} d\, c_{\mathrm{c}}+a\, \|\mathbf{H}\|\sqrt{N}c_{\mathrm{o}}\right)\,\frac{t^{1-\delta}}{1-\delta},\label{eqn:g_t}\\
		    \|\mathbf{x}^t-\mathbf{1}_{N}\otimes \boldsymbol{\theta}^{\ast}\|&\leq g_t'=\|\mathbf{x}^0-\mathbf{1}_{N}\otimes \boldsymbol{\theta}^{\ast}\|+\left(b\,\sqrt{MN} d\, c_{\mathrm{c}}+a\, \|\mathbf{H}\|\sqrt{N}c_{\mathrm{o}}\right)\,\frac{t^{1-\delta}}{1-\delta}.\label{eqn:g_t'}
		\end{align}
		Consequently, $\|\mathbf{H}\left(\mathbf{x}^t-\mathbf{1}_{N}\otimes \boldsymbol{\theta}^{\ast}\right)\|\leq\|\mathbf{H}\|\,g_t'.$
	\end{lemma}
	\begin{proof}
	    Using the boundness of the nonlinearities, we have that
	$\|\mathbf{L}_{\boldsymbol{\Psi}_\mathrm{c}}(\mathbf{x})\|^2\leq \sqrt{MN} d\, c_{\mathrm{c}}$ and
	$\|\mathbf{H}^{\top}\boldsymbol{\Psi}_\mathrm{o}\left(     
	\mathbf{H}\left(\mathbf{1}_{N}\otimes \boldsymbol{\theta}^{\ast}-\mathbf{x}^t\right)+\mathbf{n}^{t}
	\right)\|\leq \|\mathbf{H}\|\sqrt{N}c_{\mathrm{o}},$
	  where $d=\max\limits_{i}d_i.$ Therefore, recalling the algorithm~\eqref{eq:algonoise}, for all $t>0$ we have that 
		\begin{align*}
		    \|\mathbf{x}^t\|&\leq \|\mathbf{x}^{t-1}\|+\alpha_{t-1}\underbrace{\left(\frac{b}{a}\sqrt{MN} d\, c_{\mathrm{c}}+ \|\mathbf{H}\|\sqrt{N}c_{\mathrm{o}}\right)}_{c}
		    \leq \|\mathbf{x}^{t-2}\|+\alpha_{t-2}\,c+\alpha_{t-1}\,c\\
		    &\leq \|\mathbf{x}^0\|+c\,\sum\limits_{j=0}^{t-1}\frac{a}{(1+j)^\delta}\leq\|\mathbf{x}^0\|+c\,\int\limits_{0}^{t-1}\frac{a}{(1+s)^\delta}ds\leq \|\mathbf{x}^0\|+c\,a\,\frac{t^{1-\delta}}{1-\delta}.
		\end{align*}
	Analogously, for all $t>0$, we have that $\|\mathbf{x}^t-\mathbf{1}_{N}\otimes \boldsymbol{\theta}^{\ast}\|\leq g_t',$ and as a consequence $\|\mathbf{H}\left(\mathbf{x}^t-\mathbf{1}_{N}\otimes \boldsymbol{\theta}^{\ast}\right)\|\leq\|\mathbf{H}\|\,g_t'.$ 
	\end{proof}

Next, we have the following Lemma that bounds quantities $T_1(x)$ and $T_2(x).$
	
	\begin{lemma}\label{lemma:T1T2}
	    Let Assumptions \ref{as:newtworkmodelandobservability}-\ref{as:nonlinearity} hold. Then, for the sequence of iterates $\{\mathbf{x}^t\}$ 
		generated by algorithm~\eqref{eq:algcomp}, provided that the step-size sequence $\{\alpha_t\}$ is given by $\alpha_t=a/(t+1)^\delta,$ $a>0,\delta\in(0.5,1)$, we have that there exist positive constants $G_c$ and $G_0$ such that, for any outcome $\omega$:
		\begin{align*}
            T_1(\mathbf{x}^t)&\geq\frac{\varphi_c'(0)G_{\mathrm{c}}}{4g_t}\left(\mathbf{x}^t-\mathbf{1}_{N}\otimes \boldsymbol{\theta}^{\ast}\right)^\top \mathbf{L}\otimes\mathbf{I} \left(\mathbf{x}^t-\mathbf{1}_{N}\otimes \boldsymbol{\theta}^{\ast}\right),\\
            T_2(\mathbf{x}^t)&\geq     \frac{\varphi_o'(0)G_o}{2\|\mathbf{H}\|g_t'}
            \left(\mathbf{x}^t-\mathbf{1}_{N}\otimes \boldsymbol{\theta}^{\ast}\right)^\top\mathbf{H}^\top\mathbf{H}\left(\mathbf{x}^t-\mathbf{1}_{N}\otimes \boldsymbol{\theta}^{\ast}\right),
        \end{align*}
	\end{lemma}
	To prove Lemma~\ref{lemma:T1T2}, we make use of the following Lemma from~\cite{gardientHTnoise} (see Lemma 5.5 in~\cite{gardientHTnoise}).
    \begin{lemma}\label{lemma:bound}
        Consider function $\varphi$ in \eqref{eq:phi-def}, there exists a positive constant $G$ such that $|\varphi(a)|\leq \frac{\varphi'(0)\, G |a|}{2\,g},$ for all $|a|<g.$
    \end{lemma}
    \begin{proof} (Proof of Theorem~\ref{lemma:T1T2})
    Using Lemma~\ref{lemma:bound} for function $\varphi_c$ we get that there exists a positive constant $G_{\mathrm{c}}$ such that 
    \begin{align}
        T_1(\mathbf{x}^t)&=\sum_{\{i,j\}\in E,\, i<j}\left(\mathbf{x}^t_{i}-\mathbf{x}^t_{j}\right)^{\top}\boldsymbol{\varphi}_\mathrm{c}\left(\mathbf{x}^t_{i}-\mathbf{x}^t_{j}\right)\nonumber\\&=\sum_{\{i,j\}\in E,\, i<j}\sum\limits_{\ell=1}^M\left((\mathbf{x}^t_{i})_\ell-(\mathbf{x}^t_{j})_\ell\right)^{\top}\boldsymbol{\varphi}_\mathrm{c}\left((\mathbf{x}^t_{i})_\ell-(\mathbf{x}^t_{j})_\ell\right)\nonumber\\
        \label{eqn:Gc}&\geq \frac{\varphi_c'(0)G_{\mathrm{c}}}{4g_t} \sum_{\{i,j\}\in E,\, i<j}\|\mathbf{x}^t_i-\mathbf{x}^t_j\|^2=\frac{\varphi_c'(0)G_{\mathrm{c}}}{4g_t}(\mathbf{x}^t)^\top (\mathbf{L}\otimes\mathbf{I}) \mathbf{x}^t\\
        &=\frac{\varphi_c'(0)G_{\mathrm{c}}}{4g_t}\left(\mathbf{x}^t-\mathbf{1}_{N}\otimes \boldsymbol{\theta}^{\ast}\right)^\top \mathbf{L}\otimes\mathbf{I} \left(\mathbf{x}^t-\mathbf{1}_{N}\otimes \boldsymbol{\theta}^{\ast}\right),\nonumber
    \end{align}
    since, from Lemma~\ref{lemma:boundsonx} we have $\|\mathbf{x}^t\|\leq g_t.$
    Analogously, from Lemma~\ref{lemma:bound} we have that for the function $\varphi_o$ there exists a positive constant $G_o$ such that
    \begin{align}
        T_2(\mathbf{x})&=\sum\limits_{i=1}^N \left(
		\mathbf{H}\left(\mathbf{x}-\mathbf{1}_{N}\otimes \boldsymbol{\theta}^{\ast}\right)
		\right)_i{\varphi}_\mathrm{o}\left(\left(
		\mathbf{H}\left(\mathbf{x}-\mathbf{1}_{N}\otimes \boldsymbol{\theta}^{\ast}\right)
		\right)_i\right)\nonumber\\
		\label{eqn:Go}&\geq \frac{\varphi_o'(0)G_o}{2\|\mathbf{H}\|g_t'}\left(\mathbf{x}-\mathbf{1}_{N}\otimes \boldsymbol{\theta}^{\ast}\right)^\top\mathbf{H}^\top\mathbf{H}\left(\mathbf{x}-\mathbf{1}_{N}\otimes \boldsymbol{\theta}^{\ast}\right),
    \end{align}
    since, from Lemma~\ref{lemma:boundsonx} we have  $\|\mathbf{H}\left(\mathbf{x}^t-\mathbf{1}_{N}\otimes \boldsymbol{\theta}^{\ast}\right)\|\leq\|\mathbf{H}\|\,g_t'.$ 
    \end{proof}
    
    We next have the following theorem that analyzes positive definiteness of the matrix $\frac{\varphi_c'(0)G_{\mathrm{c}}}{4g_t}\mathbf{L}\otimes\mathbf{I}+
         \frac{\varphi_o'(0)G_o}{2\|\mathbf{H}\|g_t'} \mathbf{H}^\top\mathbf{H}.$
         
    \begin{lemma}\label{lemma:pdmatrix}
        Let Assumptions \ref{as:newtworkmodelandobservability}-\ref{as:nonlinearity} hold. The following is true for any $\mathbf{x} \in {\mathbb R}^{MN}$: 
        \begin{align*}
         &\left(\mathbf{x}-\mathbf{1}_{N}\otimes \boldsymbol{\theta}^{\ast}\right)^\top\left(    
         \frac{\varphi_c'(0)G_{\mathrm{c}}}{4g_t}\mathbf{L}\otimes\mathbf{I}+
         \frac{\varphi_o'(0)G_o}{2\|\mathbf{H}\|g_t'} \mathbf{H}^\top\mathbf{H}\right)\left(\mathbf{x}-\mathbf{1}_{N}\otimes \boldsymbol{\theta}^{\ast}\right)\\&\geq\min\left\{\frac{\varphi_o'(0)G_o}{2\|\mathbf{H}\|g_t'} \left(\frac{\lambda_{\mathrm{H}}}{N}-\frac{2S_{\mathrm{H}}}{\sqrt{N}}k\right),\frac{b\,\varphi_c'(0)G_{\mathrm{c}}}{4ag_t}\frac{\lambda_2(\mathbf{L})}{1+\frac{1}{k^2}}\right\}\|\mathbf{x}-\mathbf{1}_{N}\otimes \boldsymbol{\theta}^{\ast}\|^2, 
     \end{align*}
     where $g_t$ and $g_t'$ are defined in Lemma~\ref{lemma:boundsonx}, $G_c$ and $G_o$ in Lemma~\ref{lemma:T1T2},  $S_{\mathrm{H}}=\sum\limits_{i=1}^N \|\mathbf{h}_i\|^2,$ $\lambda_{\mathrm{H}}=\lambda_1\left(\sum\limits_{i=1}^N  \mathbf{h}_i \mathbf{h}_i^\top\right)>0$ is the smallest eigenvalue of regular matrix $\sum\limits_{i=1}^N  \mathbf{h}_i \mathbf{h}_i^\top$ (see Assumption~\ref{as:newtworkmodelandobservability}) and recalling that $\lambda_2(\mathbf{L})>0$ is the smallest positive eigenvalue of Laplacian matrix $\mathbf{L}.$
    \end{lemma}
    \begin{proof}
        Let us consider matrix $\mathbf{L}\otimes\mathbf{I}+\mathbf{H}^\top\mathbf{H}$ and follow argument as in Appendix~A of~\cite{EXTRA}. For any $\mathbf{x}\in\mathbb{R}^{MN}$, we have that there exist vectors $\mathbf{u}\in \text{span}\{\mathbf{1}\otimes \mathbf{m}| \mathbf{m}\in\mathbb{R}^M\}$ and $\mathbf{v}\in \text{span}\{\mathbf{1}\otimes \mathbf{m}| \mathbf{m}\in\mathbb{R}^M\}^\perp$ such that $\mathbf{x}=\mathbf{u}+\mathbf{v}.$ Firstly, we have that
    \begin{align*}
        \left(\mathbf{u}-\mathbf{1}_{N}\otimes \boldsymbol{\theta}^{\ast}\right)^\top\mathbf{H}^\top\mathbf{H}\left(\mathbf{u}-\mathbf{1}_{N}\otimes \boldsymbol{\theta}^{\ast}\right)&=\sum\limits_{i=1}^N (\hat{\mathbf{u}}-\boldsymbol{\theta}^\star)^\top \mathbf{h}_i \mathbf{h}_i^\top(\hat{\mathbf{u}}-\boldsymbol{\theta}^\star)\\
        &=(\hat{\mathbf{u}}-\boldsymbol{\theta}^\star)^\top\left(\sum\limits_{i=1}^N  \mathbf{h}_i \mathbf{h}_i^\top\right)(\hat{\mathbf{u}}-\boldsymbol{\theta}^\star)\\
        &\geq \lambda_{\mathrm{H}}\|\hat{\mathbf{u}}-\boldsymbol{\theta}^\star\|^2,
    \end{align*}
    where $\hat{\mathbf{u}}\in\mathbb{R}^M$ such that $\mathbf{u}=\mathbf{1}\otimes\hat{\mathbf{u}}.$ Notice here that $\|\mathbf{u}-\mathbf{1}_{N}\otimes \boldsymbol{\theta}^{\ast}\|=\sqrt{N}\|\hat{\mathbf{u}}-\boldsymbol{\theta}^\star\|.$ Secondly, $\left(\mathbf{x}-\mathbf{u}\right)^\top\mathbf{H}^\top\mathbf{H}\left(\mathbf{x}-\mathbf{u}\right)\geq0,$ since $\mathbf{H}^\top\mathbf{H}$ is positive semi-definite matrix. Thirdly, following also holds
    \begin{align*}
        \left(\mathbf{u}-\mathbf{1}_{N}\otimes \boldsymbol{\theta}^{\ast}\right)^\top\mathbf{H}^\top\mathbf{H}\left(\mathbf{x}-\mathbf{u}\right)&=\sum\limits_{i=1}^N(\hat{\mathbf{u}}-\boldsymbol{\theta}^\star)^\top \mathbf{h}_i\mathbf{h}_i^\top (\mathbf{x}_i-\hat{\mathbf{u}})\\
        &\geq-\sum\limits_{i=1}^N \|\hat{\mathbf{u}}-\boldsymbol{\theta}^\star\| \|\mathbf{h}_i\|^2\|\mathbf{x}_i-\hat{\mathbf{u}}\|\\
        &\geq -\|\hat{\mathbf{u}}-\boldsymbol{\theta}^\star\|\|\mathbf{v}\|S_{\mathrm{H}}.
    \end{align*}
   Analogously, we have that $\left(\mathbf{x}-\mathbf{u}\right)^\top\mathbf{H}^\top\mathbf{H}\left(\mathbf{u}-\mathbf{1}_{N}\otimes \boldsymbol{\theta}^{\ast}\right)\geq -\|\hat{\mathbf{u}}-\boldsymbol{\theta}^\star\|\|\mathbf{v}\|S_{\mathrm{H}}.$ Therefore,
    \begin{align*}
        \left(\mathbf{x}-\mathbf{1}_{N}\otimes \boldsymbol{\theta}^{\ast}\right)^\top\mathbf{H}^\top\mathbf{H}\left(\mathbf{x}-\mathbf{1}_{N}\otimes \boldsymbol{\theta}^{\ast}\right)\geq \lambda_{\mathrm{H}}\|\hat{\mathbf{u}}-\boldsymbol{\theta}^\star\|^2-2S_{\mathrm{H}}\|\hat{\mathbf{u}}-\boldsymbol{\theta}^\star\|\|\mathbf{v}\|.
    \end{align*}
    We also have that $\mathbf{u}-\mathbf{1}\otimes\boldsymbol{\theta}^\star\in\text{null}(\mathbf{L}\otimes\mathbf{I})$ and $\mathbf{v}\in\text{Range}(\mathbf{L}\otimes\mathbf{I})$ and, hence, we have that
    \begin{align*}
        \left(\mathbf{x}-\mathbf{1}_{N}\otimes \boldsymbol{\theta}^{\ast}\right)^\top\mathbf{L}\otimes\mathbf{I}\left(\mathbf{x}-\mathbf{1}_{N}\otimes \boldsymbol{\theta}^{\ast}\right)&=\left(\mathbf{u}-\mathbf{1}_{N}\otimes \boldsymbol{\theta}^{\ast}+\mathbf{v}\right)^\top\mathbf{L}\otimes\mathbf{I}\left(\mathbf{u}-\mathbf{1}_{N}\otimes \boldsymbol{\theta}^{\ast}+\mathbf{v}\right)\\
        &=\mathbf{v}^\top\mathbf{L}\otimes\mathbf{I}\,\mathbf{v}\geq \lambda_2(\mathbf{L}\otimes\mathbf{I})\|\mathbf{v}\|^2=\lambda_2(\mathbf{L})\|\mathbf{v}\|^2.
    \end{align*}
    Let $k>0$ be arbitrarily chosen. If $\|\mathbf{v}\|\leq k \|\mathbf{u}-\mathbf{1}_{N}\otimes \boldsymbol{\theta}^{\ast}\|$, then we have that
    \begin{align*}
        \left(\mathbf{x}-\mathbf{1}_{N}\otimes \boldsymbol{\theta}^{\ast}\right)^\top&\left(\mathbf{L}\otimes\mathbf{I}+\mathbf{H}^\top\mathbf{H}\right)\left(\mathbf{x}-\mathbf{1}_{N}\otimes \boldsymbol{\theta}^{\ast}\right)\\&\geq \lambda_{\mathrm{H}}\|\hat{\mathbf{u}}-\boldsymbol{\theta}^\star\|^2-2S_{\mathrm{H}}\|\hat{\mathbf{u}}-\boldsymbol{\theta}^\star\|\|\mathbf{v}\|+\lambda_2(\mathbf{L})\|\mathbf{v}\|^2\\
        &\geq \left( \frac{\lambda_{\mathrm{H}}}{N}-\frac{2S_{\mathrm{H}}}{\sqrt{N}}k\right)\|\mathbf{u}-\mathbf{1}_{N}\otimes \boldsymbol{\theta}^{\ast}\|^2+\lambda_2(\mathbf{L})\|\mathbf{v}\|^2\\
        &\geq \min\{ \frac{\lambda_{\mathrm{H}}}{N}-\frac{2S_{\mathrm{H}}}{\sqrt{N}}k,\lambda_2(\mathbf{L})\}\|\mathbf{x}-\mathbf{1}_{N}\otimes \boldsymbol{\theta}^{\ast}\|^2,
    \end{align*}
     where in the last inequality we used the fact that $\|\mathbf{x}-\mathbf{1}_{N}\otimes \boldsymbol{\theta}^{\ast}\|^2=\|\mathbf{u}-\mathbf{1}_{N}\otimes \boldsymbol{\theta}^{\ast}\|^2+\|\mathbf{v}\|^2.$ If $\|\mathbf{v}\|\geq k \|\mathbf{u}-\mathbf{1}_{N}\otimes \boldsymbol{\theta}^{\ast}\|,$ then
     
     \begin{align*}
         \left(\mathbf{x}-\mathbf{1}_{N}\otimes \boldsymbol{\theta}^{\ast}\right)^\top&\left(\mathbf{L}\otimes\mathbf{I}+\mathbf{H}^\top\mathbf{H}\right)\left(\mathbf{x}-\mathbf{1}_{N}\otimes \boldsymbol{\theta}^{\ast}\right)\geq 0 +\lambda_2(\mathbf{L})\|\mathbf{v}\|^2\\
         &\geq\frac{\lambda_2(\mathbf{L})}{1+\frac{1}{k^2}}\|\mathbf{v}\|^2+\frac{\lambda_2(\mathbf{L})}{1+\frac{1}{k^2}}\|\mathbf{u}-\mathbf{1}_{N}\otimes \boldsymbol{\theta}^{\ast}\|^2\\
         &\geq \frac{\lambda_2(\mathbf{L})}{1+\frac{1}{k^2}}\|\mathbf{x}-\mathbf{1}_{N}\otimes \boldsymbol{\theta}^{\ast}\|^2.
     \end{align*}
     
     Therefore, regardless of vector $\mathbf{v},$ we have that 
     \begin{align*}
         \left(\mathbf{x}-\mathbf{1}_{N}\otimes \boldsymbol{\theta}^{\ast}\right)^\top&\left(\mathbf{L}\otimes\mathbf{I}+\mathbf{H}^\top\mathbf{H}\right)\left(\mathbf{x}-\mathbf{1}_{N}\otimes \boldsymbol{\theta}^{\ast}\right)\\&\geq \min\left\{ \frac{\lambda_{\mathrm{H}}}{N}-\frac{2S_{\mathrm{H}}}{\sqrt{N}}k,\frac{\lambda_2(\mathbf{L})}{1+\frac{1}{k^2}}\right\}\|\mathbf{x}-\mathbf{1}_{N}\otimes \boldsymbol{\theta}^{\ast}\|^2.
     \end{align*}
     Following the same idea, we get that
     \begin{align}
         &\left(\mathbf{x}-\mathbf{1}_{N}\otimes \boldsymbol{\theta}^{\ast}\right)^\top\left(    
         \frac{\varphi_c'(0)G_{\mathrm{c}}}{4g_t}\mathbf{L}\otimes\mathbf{I}+
         \frac{\varphi_o'(0)G_o}{2\|\mathbf{H}\|g_t'} \mathbf{H}^\top\mathbf{H}\right)\left(\mathbf{x}-\mathbf{1}_{N}\otimes \boldsymbol{\theta}^{\ast}\right)\nonumber\\&\geq\min\left\{\frac{\varphi_o'(0)G_o}{2\|\mathbf{H}\|g_t'} \left(\frac{\lambda_{\mathrm{H}}}{N}-\frac{2S_{\mathrm{H}}}{\sqrt{N}}k\right),\frac{b\,\varphi_c'(0)G_{\mathrm{c}}}{4ag_t}\frac{\lambda_2(\mathbf{L})}{1+\frac{1}{k^2}}\right\}\|\mathbf{x}-\mathbf{1}_{N}\otimes \boldsymbol{\theta}^{\ast}\|^2.\label{eqn:boundMSE}
     \end{align}
    \end{proof}
    Finally, to prove Theorem~\ref{theorem-MSE}, we make use of the following Lemma from~\cite{gardientHTnoise} (see Theorem 5.2 in~\cite{gardientHTnoise}).
    \begin{lemma}\label{lemma:sequence}
        Let $z^t$ be a nonnegative (deterministic) sequence satisfying
        \begin{align*}
            z^{t+1}\leq(1-r_1^t)z^t+r_2^t,
        \end{align*}
        for all $t\geq t',$ for some $t'>0,$ with some $z^{t'}\geq0.$ Here, $\{r_1^t\}$ and $\{r_2^t\}$ are deterministic sequences with $\frac{a_1}{t+1} \leq r_1^t\leq1$ and $r_2^t\leq\frac{a_2}{(t+1)^{\delta}},$ with $a_1,a_2>0,$ and $\delta>0.$ Then, the following holds: (1) $z^t=O(\frac{1}{t^{\delta-1}})$ provided that $a_1>\delta-1;$ (2) if $a_1\leq\delta-1,$ them $z^t=O(\frac{1}{t^s}),$ for any $s<a_1.$
    \end{lemma}
	
	We are finally ready to finalize the proof 
	of Theorem~\ref{theorem-MSE}.
     \begin{proof} (Proof of Theorem~\ref{theorem-MSE})
         From equations~\eqref{eqn:boundMSE} and~\eqref{eqn:boundT1T2} we get that
         \begin{align*}
             (\mathbf{x}-&\mathbf{1}_{N}\otimes \boldsymbol{\theta}^{\ast})^\top\mathbf{r}(\mathbf{x}^t)\\&\leq -\min\left\{\frac{\varphi_o'(0)G_o}{2\|\mathbf{H}\|g_t'} \left(\frac{\lambda_{\mathrm{H}}}{N}-\frac{2S_{\mathrm{H}}}{\sqrt{N}}k\right),\frac{b\,\varphi_c'(0)G_{\mathrm{c}}}{4ag_t}\frac{\lambda_2(\mathbf{L})}{1+\frac{1}{k^2}}\right\}\|\mathbf{x}-\mathbf{1}_{N}\otimes \boldsymbol{\theta}^{\ast}\|^2.
         \end{align*}
         Therefore, taking the expectation in~\eqref{eq:expfilt}, we have that 
         \begin{align*}
             \mathbb{E}[V(\mathbf{x^{t+1}})]\leq \left(1 - \frac{a_1}{t+1}\right) \mathbb{E}[V(\mathbf{x^{t}})]+\frac{a_2}{(1+t)^{2\delta}},
         \end{align*}
         where 
         \begin{align*}
             a_1=\min &\left\{ \frac{\varphi_o'(0)G_o a (1-\delta)\left(\lambda_{\mathrm{H}}-2S_{\mathrm{H}}\sqrt{N}k\right)}{\|\mathbf{H}\| N\left(\|\mathbf{x}^0-\mathbf{1}_{N}\otimes \boldsymbol{\theta}^{\ast}\|+b\,\sqrt{MN} d\, c_{\mathrm{c}}+a\, \|\mathbf{H}\|\sqrt{N}c_{\mathrm{o}}\right)},\right.\\&\left.
         \frac{b\,\varphi_c'(0)G_{\mathrm{c}}(1-\delta)\lambda_2(\mathbf{L})k^2}{2(k^2+1)\left(\|\mathbf{x}^0\|+b\,\sqrt{MN} d\, c_{\mathrm{c}}+a\, \|\mathbf{H}\|\sqrt{N}c_{\mathrm{o}}\right)}\right\}
         \end{align*}
         and $a_2=a^2c'.$ Therefore, using the Lemma~\ref{lemma:sequence}, $\hat{\delta}$ is any positive number such that 
		\begin{align*}
		    \hat{\delta}<\min\left\{2\delta-1, \frac{\varphi_o'(0)G_{\mathrm{o}} a (1-\delta)\left(\lambda_{\mathrm{H}}-2S_{\mathrm{H}}\sqrt{N}k\right)}{\|\mathbf{H}\| N\left(\|\mathbf{x}^0-\mathbf{1}_{N}\otimes \boldsymbol{\theta}^{\ast}\|+b\,\sqrt{MN} d\, c_{\mathrm{c}}+a\, \|\mathbf{H}\|\sqrt{N}c_{\mathrm{o}}\right)},\right.\\ \left.
         \frac{b\,\varphi_c'(0)G_{\mathrm{c}}(1-\delta)\lambda_2(\mathbf{L})k^2}{2(k^2+1)\left(\|\mathbf{x}^0\|+b\,\sqrt{MN} d\, c_{\mathrm{c}}+a\, \|\mathbf{H}\|\sqrt{N}c_{\mathrm{o}}\right)}\right\}.
		\end{align*}
         Therefore, using Lemma~\ref{lemma:sequence} we obtain MSE convergence with rate $O(1/t^{\hat{\delta}}).$
     \end{proof}
     \begin{remark}
     Even though, we see that the convergence factor $\hat{\delta}$ depends on the system parameters, i.e., on the network and sensing model and also on the innovation and consensus nonlinearities, it is easy to see that $\hat{\delta}\in(0,1)$ regardless of the system parameters. 
     {Recall that Theorem~\ref{theorem-asymptotic-normality} shows that the proposed estimator~\eqref{eq:alg2} obtains rate $1/t$ in the \emph{weak convergence sense}, while Theorem~\ref{theorem-MSE} shows that  \eqref{eq:alg2} obtains a slower convergence rate, but in the sense of the \emph{mean squared convergence}. Note that this is not a contradiction, and Theorem~\ref{theorem-MSE} adds information with respect to Theorem~\ref{theorem-asymptotic-normality}. Namely, it is well known that mean squared convergence implies convergence in distribution; therefore, with the same assumptions as in Theorem~\ref{theorem-MSE}, the convergence rate ${1}/{t^{\hat{\delta}}}$ is also attainable for convergence in distribution. In contrast, from Theorem~\ref{theorem-asymptotic-normality},  we can not conclude that the rate of the mean squared convergence is also ${1}/{t}.$}
     \end{remark}
     \begin{remark}\label{remark:meanestimator}
     {
     In fact, we next show that, in the presence of the heavy-tailed observation noise considered here, the MSE convergence rate cannot be as fast as $1/t$, for any estimator (even not for centralized ones). In this sense, 
     the fact that quantity $\hat{\delta}$ is strictly smaller than one is not a consequence of loose bounds, but it is rather due to the intrinsic difficulty of the estimation problem. To be specific, we consider here the special case where each agent $i$ observes a scalar parameter $\theta^\star\in\mathbb{R}$ according to
	\begin{align}\label{eq:meanestimator}
		z_i(t)=\theta^\star + n_i^t,
	\end{align}
    where $n_i^t$ satisfies Assumption~\ref{as:observationnoise}. In this case, the proposed estimator~\eqref{eq:alg2} can be viewed as a mean estimator of the probability density function $p_\mathrm{o}(u-\theta^\star)$. Let us denote by $\mathcal{P}$ the class of all probability density functions $p_\mathrm{o}(u-\theta^\star)$ such that $p_\mathrm{o}$ is the pdf of the observation noise that satisfies Assumption~\ref{as:observationnoise}, for any $\theta^\star\in\mathbb{R}$. Extending the results from~\cite{Sub_Gaussian} (see Appendix G), we prove that, for any $\theta^\star\in\mathbb{R}$, and for any mean estimator $\hat{\theta}_t$, the following holds:
\begin{align}\label{eq:subgassianestimator}
    \sup\limits_t \sup\limits_{p\in\mathcal{P}} tN \mathbb{E}[|\hat{\theta}_t-\theta^\star|^2]=+\infty.
\end{align} 
On the other hand, Theorem~\ref{theorem-MSE} shows that, with the proposed distributed estimator~\eqref{eq:alg2}, the following holds:
\begin{align*}
    \sup\limits_t \sup\limits_{p\in\mathcal{P}} (tN)^{\hat{\delta}}  \mathbb{E}[|\hat{\theta}_t-\theta^\star|^2]<+\infty,
\end{align*}
for some $\hat{\delta}\in(\frac{1}{2},1)$\footnote{{Notice that in the centralized case, the observations are collected in batches of fixed size $N$. That is, after $t$ time steps, there are $Nt$ observations. Henceforth, we include  quantity $N$ in~\eqref{eq:subgassianestimator} for a precise statement. Note that, since $N$ is  constant and the supremum is taken with respect to $t$, the inclusion of $N$ is not necessary.}}
}
     \end{remark}

     \begin{remark}\label{remark:regressors}
         \noindent{Theorems~\ref{theorem-almost-surely}, \ref{theorem-asymptotic-normality} and~\ref{theorem-MSE} continue to hold even if the linear transformation vectors $\mathbf{h}_i$ in \eqref{eq:obs_model} are no longer static (see Appendix H). That is, we can allow that each agent $i$ at each time $t=0,1,...,$ makes the observation by:
	\begin{align}\label{eq:obsregressor}
	    z_i^t=(\mathbf{h}_i^t)^\top \boldsymbol{\theta}^\star+n_i^t.
	\end{align}
	Here, for each agent $i$, for each time step $t$, the linear transformation vector $\mathbf{h}_i^t$ is a random variable that satisfies the following assumptions.}
{	\begin{enumerate}
	    \item For each agent $i$ and each time step $t=0,1,...,$, the linear transformation vector is given by $\mathbf{h}_i^t=\overline{\mathbf{h}}_i+\widetilde{\mathbf{h}}_i^t,$
	    where the vector $\overline{\mathbf{h}}_i\in\mathbf{R}^M$ is deterministic, and vector $\widetilde{\mathbf{h}}_i^t\in\mathbf{R}^M$ is a random vector;
        \item The sequence of vectors $\{[\widetilde{\mathbf{h}}_1^t,\widetilde{\mathbf{h}}_2^t,...,\widetilde{\mathbf{h}}_N^t]\}$ is i.i.d., with finite second moment, and it is independent of the sequences $\mathbf{n}^t$ and $\boldsymbol{\xi}_{ij}^t$ for $\{i,j\}\in E$;
        \item At each agent $i=1,...,N$~at each time~$t=0,1,...,$, each entry $\ell=1,2,...M$ $[\widetilde{\mathbf{h}}_i^t]_\ell$ has the same probability density function~$p_\mathrm{h}$;
        \item The pdf $p_\mathrm{h}$ is symmetric, i.e. $p_\mathrm{h}(u)=p_\mathrm{h}(-u),$ for every $u\in\mathbb{R}$ and $p_\mathrm{h}(u)>0$ for $|u|\leq c_\mathrm{h}$, for some constant $c_\mathrm{h}>0$;
	    \item The matrix $\sum_{i=1}^{N}\overline{\mathbf{h}}_i \left( \overline{\mathbf{h}}_i\right)^\top$ is invertible.
	\end{enumerate}}
     \end{remark}

	\section{Analytical and numerical examples}	\label{section:examples}
	
	In this section we provide analytical and numerical examples that illustrate results from Section~\ref{section-theoresults}.
	
	\noindent\textbf{Example 1:} We consider the network where each agent $i$ observes a scalar parameter $\theta^\star\in\mathbb{R}$ following the linear regression model:
	\begin{align}\label{eqn:linearRegresionscalar}
		z_i(t)=h\theta^\star + n_i^t,
	\end{align}
	where $h\neq0$ and $n_i(t)$ is zero mean and i.i.d. in time and across agents. For simplicity, we assume that the underlying graph of the network is regular, with degree $d$. We assume that there is no communication noise between agents, i.e., $\boldsymbol{\xi}_{ij}\equiv0$ for $(i,j)\in E_d.$ We additionally assume that the nonlinearity on the consensus part $\Psi_{\mathrm{c}}$ in~\eqref{eq:alg2} is the identity function and the nonlinearity on the innovation part is $\Psi_{\mathrm{o}}(w)=B\tanh(w/B)$, for $B>0$. Therefore, algorithm~\eqref{eq:alg2} is now given by:
	\begin{align}\label{eq:algExp1}
		x_{i}^{t+1}={x}_{i}^{t}-\alpha_{t}\left(\frac{b}{a}\sum_{j\in\Omega_{i}}
		\left( {x}_{i}^{t}-{x}_{j}^{t} 
		\right)-{h}\Psi_\mathrm{o}\left(z_{i}^{t}-{h}{x}_{i}^{t}\right)\right),
	\end{align}
	for each agent $i$ and each time $t$. From Theorem~\ref{theorem-asymptotic-normality}, we have that the asymptotic covariance matrix is given by~\eqref{eqn-asympt-var} 
	and matrix $\mathbf{S}_0$ is now given by $\mathbf{S}_0=\sigma_\mathrm{o}^2 h^2 \mathbf{I}$ and $\sigma_\mathrm{o}^2=\int |\Psi_\mathrm{o}(w)|^2 d \Phi_\mathrm{o}(w)$ is the effective observation noise.
	Following the same procedure as in~\cite{KMR,Ourwork}, for $\boldsymbol{\Sigma}=\frac{1}{2}\mathbf{I}-a\varphi_{\mathrm{o}}'(0)h^2\mathbf{I}$, we have that the average per-agent asymptotic variance, denoted by $\sigma_B^2=\frac{1}{N}\Tr(\mathbf{S}),$ is equal to $\sigma_B^2=\frac{a^2\sigma_\mathrm{o}^2h^2}{2ah^2\varphi_{\mathrm{o}}'(0)-1},$
	for $a>\frac{1}{2h^2\varphi_{\mathrm{o}}'(0)}$ (see Appendix F). Therefore, we need to change the constant $a$ when changing $B$, i.e., we define $a=a(B)=\frac{1}{2h^2\varphi_{\mathrm{o}}'(0)(B)}+\epsilon$\footnote{$\epsilon$ is added since we need to have that $a>\frac{1}{2h^2\varphi_{\mathrm{o}}'(0)}$.}, for some constant $\epsilon>0$, we rewrite $\sigma_B^2$ as follows (see Appendix F), $\sigma_{{B}}^2=\frac{(1+2h^2\varphi_{\mathrm{o}}'(0)\epsilon)^2\sigma_\mathrm{o}^2({B})}{8h^4\varphi_{\mathrm{o}}'(0)^3\epsilon}.$
	For the nonlinearity $\Psi_{\mathrm{o}}$ that is considered here, we have that 
    $\sigma_\mathrm{o}^2=\int\limits_{-\infty}^{+\infty}B^2\tanh^2\left(\frac{w}{B}\right)f(w)dw,$ and $\varphi_{\mathrm{o}}'(0)=\int\limits_{-\infty}^{+\infty} \Psi'(w)f(w)dw=\int\limits_{-\infty}^{+\infty}\frac{1}{\cosh^2(\frac{w}{B})}f(w)dw.$
	Notice that both functions $\sigma_\mathrm{o}^2$ and $\varphi_{\mathrm{o}}'(0)$ are increasing with respect to $B$ (see Appendix F).
	Since we have that $|B^2\tanh^2(\frac{w}{B})f(w)|\leq |w^2f(w)|$ and $|\frac{1}{\cosh^2(\frac{w}{B})}f(w)|\leq |f(w)|$ for all $w\in\mathbb{R}$ and all $B>0,$ using the Lebesgue's dominated convergence theorem, we have that
	$\lim\limits_{B\to 0^+}\sigma_\mathrm{o}^2=0, \lim\limits_{B\to+\infty}\sigma_\mathrm{o}^2=\sigma^2_\eta, \lim\limits_{B\to 0^+}\varphi_{\mathrm{o}}'(0)=0, \lim\limits_{B\to +\infty}\varphi_{\mathrm{o}}'(0)=1,$
	where $\sigma^2_\eta$ is the variance of the observation noise $\eta.$ Therefore, we have that $\sigma_0^2=\lim\limits_{B\to 0^+}\sigma_{{B}}^2=+\infty$ (see Appendix F), and $\sigma_{\infty}^2=\lim\limits_{B\to +\infty}\sigma_{{B}}^2=\frac{(1+2h^2\epsilon)^2\sigma^2_\eta}{8h^4\epsilon}.$
	Suppose now that the variance of the observation noise $\eta$ is infinite, i.e. $\sigma^2_\eta=+\infty$. This means that $\sigma_{\infty}^2=+\infty$. For the continuous function $\sigma_{{B}}^2$, defined for all $B\in(0,+\infty)$, we have that $\lim\limits_{B\to 0^+}\sigma_{{B}}^2=\lim\limits_{B\to +\infty}\sigma_{{B}}^2=+\infty$. Therefore, there exists an optimal $B^\star$ such that $\sigma_{B^\star}^2=\inf\limits_{B\in(0,\infty)} \sigma_B^2$. Note that the case $B\to \infty$ corresponds to a $\mathcal{LU}$ scheme from~\cite{KMR}, while the case $B\to0$ corresponds to each agent working in isolation. Therefore, we show analytically on the simple class of nonlinearities $\Psi_\mathrm{o}$ (hyperbolic tangent), that cooperation through a nonlinear mapping $\Psi_\mathrm{o}$ strictly improves performance with respect to both using linear and non-cooperative schemes.
	
	To numerically illustrate the above results, we now consider a sensor (agents) network with $N=8$ agents, setting that the underlying topology is given by a regular graph with degree $d=3$. The true parameter is $\theta^\ast=1$, the observation parameter is $h=1$, and the observation noise for each agent's measurements has the following pdf 
	\begin{align}\label{eq:htdistribution}
		f(w)=\frac{\beta-1}{2\,(1+|w|)^\beta}, 
	\end{align}
	with $\beta=2.05$, which has an infinite variance. Recall that we assumed that there is no communication noise between agents. We set the consensus parameter as $b=1$ and the innovation parameter as $a=a(0.3)=\frac{1}{2h^2\varphi_{\mathrm{o}}'(0)(0.3)}+0.1.$ Figure~\ref{fig:SigmaBmin} shows the average per-agent asymptotic variance $\sigma_{B}^2$ versus $B$. As it can be seen, optimal $B^\star$ approximately equals $B^\star=0.65.$ Using Monte Carlo simulations, we compare numerically an estimated per-sensor MSE across iterations, for the optimal $B^\star$ and for some sub-optimal choices of $B$. We can see that the algorithm performs better for the optimal value $B^\star$ than for the other considered suboptimal choices of $B$ (see Figure \ref{fig:comparingB}), hence confirming the theory. 
	
	\begin{figure}[tbhp]
		\centering \subfloat[]{\label{fig:SigmaBmin}\includegraphics[height=40mm]{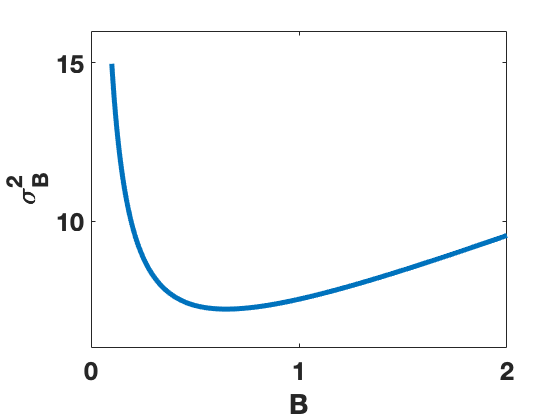}}\hspace{0.3cm} \subfloat[]{\label{fig:comparingB}\includegraphics[height=40mm]{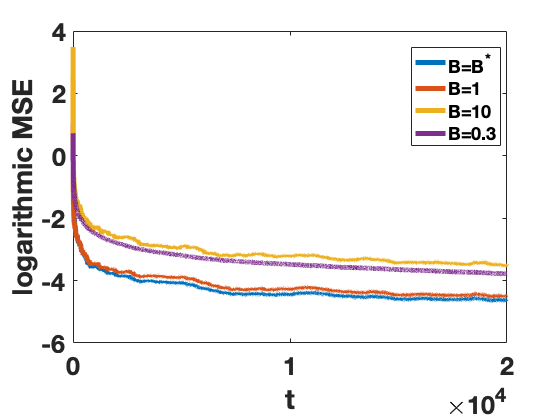}} \hspace{0.3cm} 
		\caption{(a) Average per-agent asymptotic variance $\sigma_B^2$ versus $B$ (b) Monte Carlo-estimated per-sensor MSE error on logarithmic scale for the different choices of~$B$ }
	\end{figure}

	\textbf{Example 2:} In this example we provide analysis in the terms of the average per node variance with respect to the level of the mutual dependence of observation and communication noise. Once more,
	we consider the network where each agent $i$ observes a scalar parameter $\theta^\star\in\mathbb{R}$ following the linear regression model~\eqref{eqn:linearRegresionscalar} and we assume that the underlying graph of the network is regular, with degree $d$.
	As it is said, we now allow observation and communication noise to be mutually dependent. For simplicity, we consider the case when that dependence between communication noise $\xi_{ij}^t$ and observation noise $n_i$ is given by $\xi_{ij}=\rho\, n_i^t +\sqrt{1-\rho^2}\,\hat{n}_i^t,$ at each time $t=0,1,..$ and for all tuples $\{i,j\}\in E,$ where, $\rho\in(-1,1),$ sequence $\{\hat{n}_i^t\}$ is independently identically distributed in time $t$ and across all agents $i.$ Moreover, $n_i^t$ are $\hat{n}_j^s$ mutually independent whenever $(i,t)\neq(j,s).$ Here, it is easy to see that we have strong positive correlation if $\rho\to1$, strong negative correlation if $\rho\to-1$ and we do not have any correlation if $\rho=0.$ Moreover, we set that $\Psi_{\mathrm{o}}(w)=\Psi_{\mathrm{c}}(w)=\sign w,$ and hence, algorithm~\eqref{eq:alg2} is given by
	
	\begin{align}\label{eq:algExp3}
		{x}_{i}^{t+1}={x}_{i}^{t}-\alpha_{t}\left(\frac{b}{a}\sum_{j\in\Omega_{i}}
		\Psi_\mathrm{c}\left( {x}_{i}^{t}-{x}_{j}^{t} 
		+\rho\, n_i^t +\sqrt{1-\rho^2}\,\hat{n}_i^t\right)-{h}\Psi_\mathrm{o}\left({h}(\theta^\star-{x}_{i}^{t})+n_i\right)\right).
	\end{align}
	
	Analogously to the previous example, we have that the average per-agent asymptotic variance $\sigma_{\mathrm{\rho}}^2$ is given by
	\begin{align}
	\sigma_{\mathrm{\rho}}^2=&\frac{b^2\sigma_\mathrm{c}^2d^2+a^2h^2\sigma_\mathrm{o}^2-2abhd\sigma_{\mathrm{oc}}}{N\left(2ah^2\varphi_{\mathrm{o}}'(0)-1\right)}\\&+\frac{b^2\sigma_\mathrm{c}^2d^2+a^2h^2\sigma_\mathrm{o}^2-2abhd\sigma_{\mathrm{oc}}}{N}\sum\limits_{i=2}^N\frac{1}{2b\varphi_{\mathrm{c}}'(0)\lambda_i+2ah^2\varphi_{\mathrm{o}}'(0)-1},
	\end{align}
	since $\mathbf{S}_0=\left(\frac{b^2}{a^2}\sigma_\mathrm{c}^2d^2+\sigma_\mathrm{o}^2 h^2 -2\frac{b}{a}hd\sigma_{\mathrm{oc}} \right)\mathbf{I}$ and  $\boldsymbol{\Sigma}=\frac{1}{2}\mathbf{I}-a\left(\frac{b}{a}\varphi_{\mathrm{c}}'(0)\mathbf{L}+\varphi_{\mathrm{o}}'(0)h^2\mathbf{I}\right).$ Here, regardless of $\rho$ we have that $\sigma_\mathrm{o}^2=\sigma_\mathrm{c}^2=1$ and $\varphi_{\mathrm{o}}'(0)= 2p_{n}(0)$ (see~\cite{Ourwork}).
	On the other hand, $\sigma_{\mathrm{oc}}$ which is effective cross-covariance between the observation and the communication noise after passing through the appropriate nonlinearity and $\varphi_{\mathrm{c}}'(0)$ are functions with respect to $\rho.$ We have that 
	\begin{align}
	    \sigma_{\mathrm{oc}}&=\int\limits_{-\infty}^{\infty}\int\limits_{-\infty}^{\infty}\Psi_{\mathrm{c}}(\rho x+\sqrt{1-\rho^2}y)\Psi_{\mathrm{o}}(x)p_{\hat{n}}(y) p_{n}(x)dxdy\\
	    &=\int\limits_{0}^{+\infty}\int\limits_{\frac{-\rho x}{\sqrt{1-\rho^2}}}^{\infty}p_{\hat{n}}(y)p_n(y)dydx-\int\limits_{0}^{+\infty}\int\limits_{-\infty}^{\frac{-\rho x}{\sqrt{1-\rho^2}}}p_{\hat{n}}(y)p_n(y)dydx\\
	    &-\int\limits_{-\infty}^{0}\int\limits_{\frac{-\rho x}{\sqrt{1-\rho^2}}}^{\infty}p_{\hat{n}}(y)p_n(y)dydx+
	    \int\limits_{-\infty}^{0}\int\limits_{-\infty}^{\frac{-\rho x}{\sqrt{1-\rho^2}}}p_{\hat{n}}(y)p_n(y)dydx,
	\end{align}
	and we see that $\sigma_{\mathrm{oc}}\to0$ as $\rho\to0$, $\sigma_{\mathrm{oc}}\to1$ as $\rho\to1$ and $\sigma_{\mathrm{oc}}\to-1$ as $\rho\to-1$. Moreover, we have that $\varphi_\mathrm{c}'(0)=2\int\limits_{-\infty}^{\infty}p_{\hat{n}}(-\rho x)p_n(\sqrt{1-\rho^2}x)dx,$
	and again, it is easy to see that, $\varphi_\mathrm{c}'(0)\to 2 p_n(0)$ as $\rho\to\pm1$ and $\varphi_\mathrm{c}'(0)\to 2 p_{\hat{n}}(0)$ as $\rho\to0.$
	To demonstrate the above results, again we consider a sensor (agents) network with $N=8$ agents, setting that the underlying topology is given by a regular graph with degree $d=3$. The true parameter is $\theta^\ast=1$, the observation parameter, the innovation parameter and consensus parameter are $h=a=b=1$. We set that for all $i$, $n_i$ and $\hat{n}_i$ have the pdf as in~\eqref{eq:htdistribution}
	with $\beta=2.05$. Figure~\ref{fig:sigmaro} shows $\sigma_\mathrm{\rho}^2$ with respect to $\rho$. As it can be seen, the lowest $\sigma_\mathrm{\rho}^2$ is attained at $\rho=1,$ also $\sigma_\mathrm{\rho}^2$ has two local maxima at $\rho \approx-0.88$ and at $\rho \approx 0.31.$ Figure~\ref{fig:pernodero} shows the comparison of Monte Carlo simulation for $\frac{1}{N}\|\mathbf{x}^t-\mathbf{1}\otimes\theta\|^2 \, t$ for different choices of $\rho.$ Moreover, Figure~\ref{fig:pernodero} justifies the results presented in~\ref{fig:sigmaro}, in the sense that $\frac{1}{N}\|\mathbf{x}^t-\mathbf{1}\otimes\theta\|^2 \, t$ is minimal for $\rho=1$ and maximal for $\rho=-0.88.$ 
	Finally, we note that, while the two local maxima 
	obtained here are specific for the simplistic correlation and sensing model assumed here for analytical tractability, we observe numerically 
	for more general models that the general trend of this example is preserved, in the sense that higher (more positive) correlations lead to a better performance.
	
	\begin{figure}[tbhp]
		\centering \subfloat[]{\label{fig:sigmaro}\includegraphics[height=40mm]{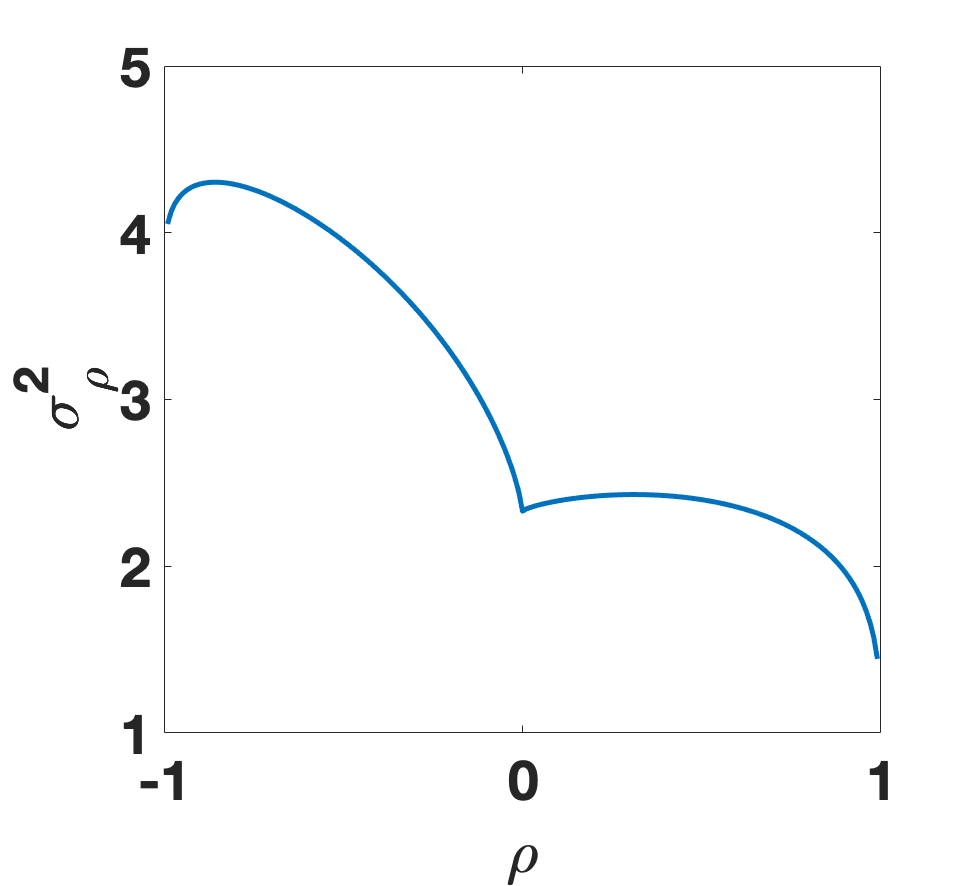}}
		\subfloat[]{\label{fig:pernodero}\includegraphics[height=40mm]{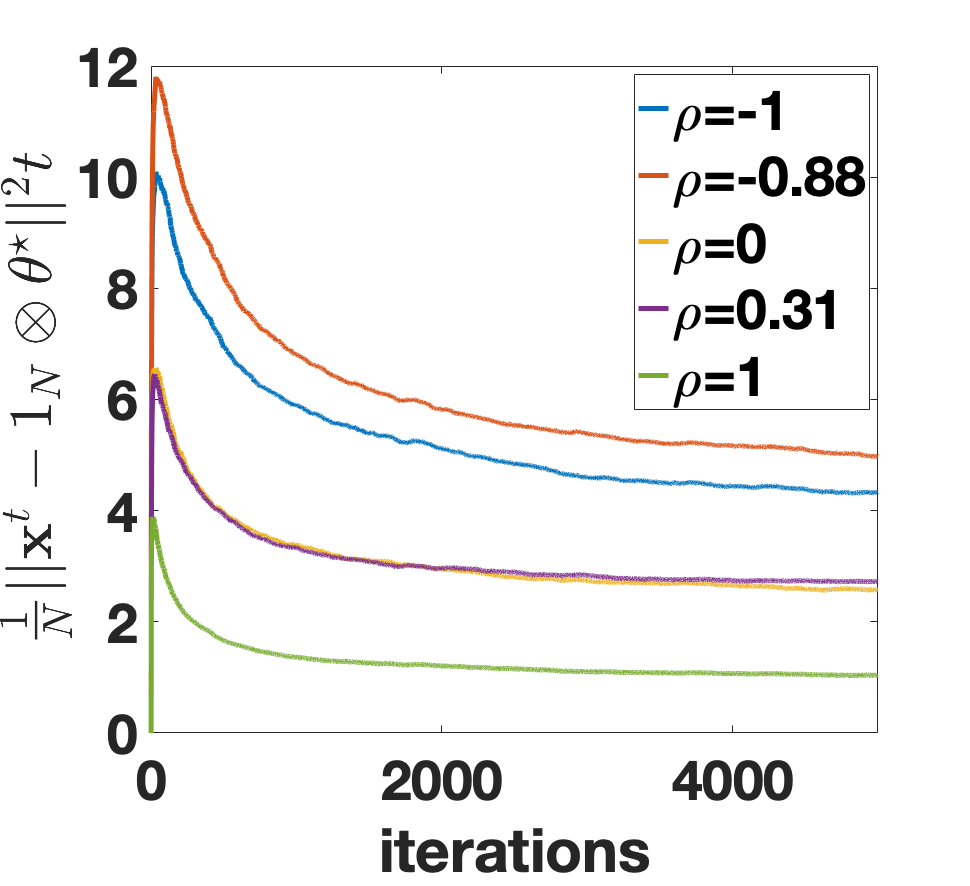}}
		\caption{(a) Average per-agent asymptotic variance $\sigma_B^2$ versus $B$  (b) Monte Carlo-estimation of $\frac{1}{N}\|\mathbf{x}^t-\mathbf{1}_N\otimes\boldsymbol{\theta}^\star\|^2\,t$ for different choices of $\rho$ } 
	\end{figure}

	\subsection{Numerical simulations}\label{subsection:numsim}

	In this subsection, we demonstrate the performance of proposed consensus+innovations estimator in a larger sensor network. We consider a sensor network with $N=40$ agents where the underlying topology is an instance of a random geometric graph;  we used randomly generated true parameter $\boldsymbol{\theta}^\star\in\mathbb{R}^{10}$, whose entries are drawn mutually independently form the uniform distribution on $[-10,10];$ we used randomly generated observation vectors $\mathbf{h}_i\in\mathbb{R}^{10}$, for which the condition $2$ of Assumption~\ref{as:newtworkmodelandobservability} is verified to be true. We set the consensus parameter as $b=1$ and and step-size parameter as $\delta=1.$   First, we compare the proposed consensus+innovations estimator with the method from~\cite{F} and its hypothetical variant in the case when there is no communication noise, but in the presence of heavy-tailed observation noise with pdf as in~\eqref{eq:htdistribution} for $\beta=2.05.$ Here, we used the same algorithm settings and the same nonlinearities for the proposed algorithm as in Example 1, with a slight change, i.e., we set that $B=10$ and $a=0.2$ For method from~\cite{F} and its hypothetical variant (see Appendix E), we set that $B_i=2$, $\phi_{i,1}(x)=x$ and $\phi_{i,2}(x)=\tanh(x)$ for all agents $i$. Furthermore, we set that weighting coefficients are chosen according to $a_{ij}=\frac{\tilde{\mathbf{A}}_{ ij}}{  \sum\limits_{\ell \in \mathcal{N}_i} \tilde{\mathbf{A}}_{\ell i}},$ where $\tilde{\mathbf{A}}=\mathbf{A}+\mathbf{I}.$ Moreover, for the smoothing recursions, zero initial conditions are assumed, $\nu_i$ is set to $0.9$ for every agent $i$ and $\epsilon=10^{-2}.$
    We can see 
	all methods manage to (slowly) decrease MSE over iterations, with the proposed method exhibiting the best performance among the three methods considered. Figure~\ref{fig:ovseovsal} shows Monte Carlo simulation of the MSE for the proposed algorithm, algorithm from~\cite{F} and the algorithm in~\cite{Ourwork}, when communication between agents is also contaminated with heavy-tailed communication noise. Here, for the proposed algorithm we set that both nonlinearities are $\Psi_{\mathrm{o}}(w)=\Psi_{\mathrm{c}}(w)=B\tanh(w/B),$ for $B=10$ and $a=1$. 
	Further, we use the same algorithm setting for the method in~\cite{F} as in the previous simulation example, and we use the same nonlinearity on the consensus part and the same $B$ for algorithm from~\cite{Ourwork} as in the proposed algorithm. 
	We can see that both~\cite{Ourwork} and~\cite{F}  
	here fail to converge, while the proposed method 
	still effectively reduces MSE.
	
	\begin{figure}[tbhp]
		\centering \subfloat[]{\label{fig:has4010}\includegraphics[height=40mm]{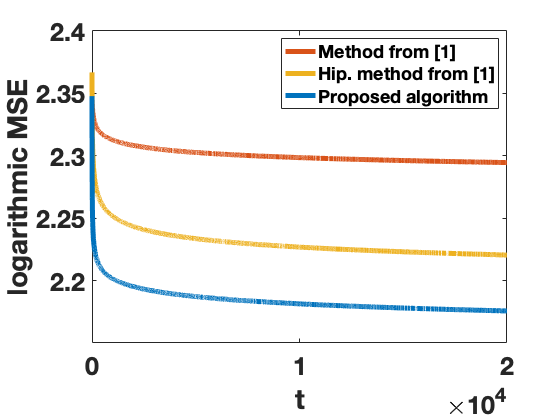}}
		\subfloat[]{\label{fig:ovseovsal}\includegraphics[height=40mm]{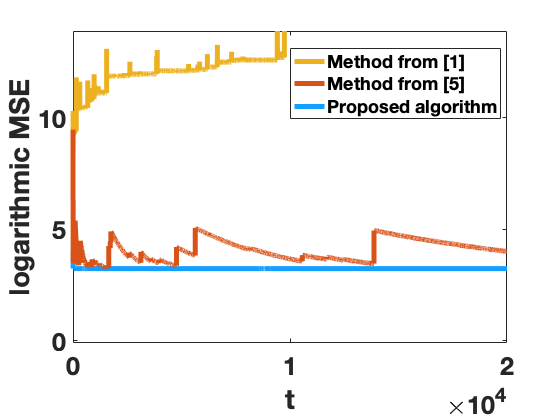}}
		\caption{(a) Monte Carlo-estimated per-sensor MSE error on logarithmic scale for proposed algorithm for $B=10$, method from~\cite{F} and its hypothetical variant (b) Monte Carlo-estimated per-sensor MSE error on logarithmic scale for proposed algorithm, algorithm form~\cite{F} and algorithm from~\cite{Ourwork}}
	\end{figure}
	
	
	We next present the scenario where the observation and communication noises are mutually dependent. To do this, we set that the $i$-th element of the observation noise $\mathbf{n}$ is given by $\mathbf{n}_i=\mathbf{v}_i\exp{(\frac{h}{2}\mathbf{v}_i^2)},$ where $\mathbf{v}$ has standard normal distribution and $h$ is a heavy-tail parameter (see \cite{Lambert}). Moreover, the $\ell$-th element of the communication noise $\boldsymbol{\xi}_{ij}$ is given by $[\boldsymbol{\xi}_{ij}]_\ell=[\mathbf{w}_{ij}]_\ell\exp{(\frac{h}{2}[\mathbf{w}_{ij}]_\ell^2)}$, where $\mathbf{w}_{ij}$ is the linear transformation of $\mathbf{v}$, i.e., $\mathbf{w}_{ij}=\mathbf{W}_{ij}\mathbf{v}$ and $\mathbf{W}_{ij}\in\mathbb{R}^{M\times N}$ is a randomly generated matrix independent of the observation noise. Figure~\ref{fig:dependent} presents Monte Carlo estimates of per-agent MSE across iterations. Figure~\ref{fig:msetD} shows Monte Carlo simulation of quantity  $\frac{1}{N}\|\mathbf{x}^t-\mathbf{1}_N\otimes\boldsymbol{\theta}^\star\|^2\,\sqrt{t}$. For this numerical setting, from the Figure~\ref{fig:msetD}, we can deduce that $E[\|\mathbf{x}^t-\mathbf{1}_N\otimes\boldsymbol{\theta}^\star\|^2]$ decreases at least as fast as $O(\frac{1}{\sqrt{t}}),$ hence confirming our MSE rate theory.

	\begin{figure}[tbhp]
		\centering  
		\subfloat[]{\label{fig:dependent}\includegraphics[height=40mm]{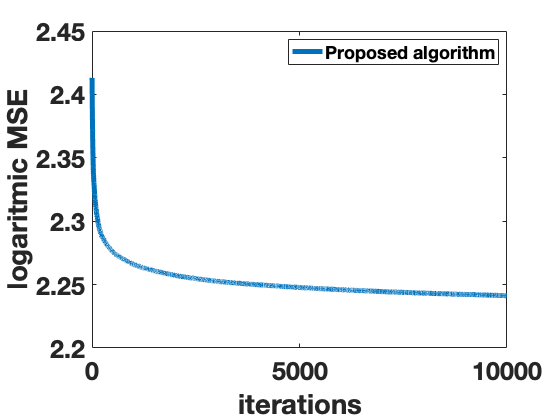}} \hspace{0.3cm} 
		\subfloat[]{\label{fig:msetD}\includegraphics[height=40mm]{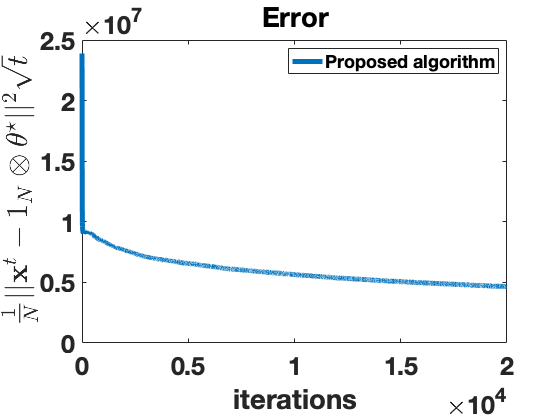}} \hspace{0.3cm} 
		\caption{(a) Monte Carlo-estimated per-sensor MSE error on logarithmic scale for proposed algorithm when link failures can occur for $B=1$ and $h=10$ (b) Monte Carlo-estimation of $\frac{1}{N}\|\mathbf{x}^t-\mathbf{1}_N\otimes\boldsymbol{\theta}^\star\|^2\,\sqrt{t}$ for $B=10$ and $h=2$.}
	\end{figure}
	
	\section{Conclusion} \label{section-conlusion}
	We have studied distributed consensus+innovations estimation under the simultaneous presence of heavy-tailed (infinite variance) correlated sensing and communication noises.
	 This setting is in contrast with existing work that either always assumes a finite-variance sensing noise. We developed a nonlinear estimator and established its almost sure convergence and asymptotic normality. Furthermore, we showed that the estimator achieves a sublinear MSE convergence rate $O(1/t^{\kappa})$, and we 
	 explicitly charaterized the rate $\kappa i\in (0,1)$ in terms of system parameters. 
	 Analytical examples illustrate the role of the nonlinearities incorporated in the method and the effects of noises correlation. Finally, numerical simulations corroborate our findings and demonstrate that the proposed distributed estimator converges  
	 under the simultaneous presence of heavy-tailed (infinite variance) correlated sensing and communication noises, while, for the same setting, existing distributed estimators fail to converge.

	\bibliographystyle{siam}
	\bibliography{references}
	
\newpage
	\section*{Appendix}
	\subsection*{A. Some results on Stochastic approximation}
	We make use of the following standard stochastic approximation result, see~\cite{Nevelson}, see also~\cite{KMR}.
	\begin{theorem}
		\label{theorem-SA}
		Let $\{\mathbf{x}^t\in\mathbb{R}^l\}_{t\geq0}$ be a random sequence:
		\begin{align}\label{eq:SAtheorem}
			\mathbf{x}^{t+1}=\mathbf{x}^t+\alpha_t[\mathbf{r}(\mathbf{x}^t)+\boldsymbol{\gamma}(t+1,\mathbf{x}^t,\omega)],
		\end{align}
		where, $\mathbf{r}(\cdot):\mathbb{R}^l\to\mathbb{R}^l$ is Borel measurable and $\{\boldsymbol{\gamma}(t,\mathbf{x},\omega)\}_{t\geq0, \mathbf{x}\in\mathbb{R}^l}$ is a family of random vectors in $\mathbb{R}^l$, defined on a probability space $(\Omega,\mathcal{F},\mathbb{P})$, and $\omega\in\Omega$ is a canonical element. Let the following sets of assumptions hold:
		\begin{itemize}
			\item[\textbf{B1:}] The function $\boldsymbol{\gamma}(t,\cdot,\cdot):\mathbb{R}^l\times\Omega \to \mathbb{R}$ is $\mathcal{B}^l\otimes\mathcal{F}$ measurable for every $t$; $\mathcal{B}^l$ is the Borel algebra of $\mathbb{R}^l$.
			\item[\textbf{B2:}] There exists a filtration $\{\mathcal{F}_t\}_{t\geq0}$ of $\mathcal{F}$, such that, for each $t$, the family of random vectors $\{\boldsymbol{\gamma}(t,\mathbf{x},\omega)\}_{\mathbf{x}\in\mathbb{R}^l}$ is $\mathcal{F}_t$ measurable, zero-mean and independent of $\mathcal{F}_{t-1}.$
		\end{itemize}
		(If Assumtions B1, B2 hold, $\{\mathbf{x}(t)\}_{t\geq0}$, is Markov.)
		\begin{itemize}
			\item[\textbf{B3:}] There exists a twice continuously differentiable $V(\mathbf{x})$ with bounded second order partial derivatives and a point $\mathbf{x}^\ast\in\mathbb{R}^l$ satisfying
			\begin{align*}
				&V(\mathbf{x}^\ast)=0, V(\mathbf{x})>0, \mathbf{x}\neq\mathbf{x}^\ast, \lim\limits_{||\mathbf{x}||\to\infty}V(\mathbf{x})=\infty,\\
				&\sup\limits_{\epsilon<||\mathbf{x}-\mathbf{x}^\ast||<\frac{1}{\epsilon}}\langle\mathbf{r}(\mathbf{x}),\nabla V(\mathbf{x})\rangle<0, \forall\epsilon>0.
			\end{align*}
			\item[\textbf{B4:}] There exists constants $k_1, k_2>0$, such that,
			\begin{align*}
				||\mathbf{r}(\mathbf{x})||^2+\mathbb{E}[||\boldsymbol{\gamma}(t+1,\mathbf{x},\omega)||^2]\leq k_1(1+V(\mathbf{x}))- k_2\langle\mathbf{r}(\mathbf{x}),\nabla V(\mathbf{x})\rangle
			\end{align*}
			\item[\textbf{B5:}] The weight sequence $\{\alpha(t)\}_{t\geq0}$ satisfies
			\begin{align*}
				\alpha_t>0, \sum\limits_{t\geq0}\alpha_t=\infty, \sum\limits_{t\geq0}\alpha^2_t<\infty.
			\end{align*}
			\item[\textbf{C1:}] The function $\mathbf{r}(\mathbf{x})$ admits the representation
			\begin{align}
				\label{eq:C.1.1}
				\mathbf{r}(\mathbf{x})=\mathbf{B}(\mathbf{x}-\mathbf{x}^\ast)+\boldsymbol{\delta}(\mathbf{x}),
			\end{align}
			where
			\begin{align}\label{eq:C.1.2}
				\lim\limits_{\mathbf{x}\to\mathbf{x}^\ast}\frac{||\boldsymbol{\delta}(\mathbf{x})||}{||\mathbf{x}-\mathbf{x}^\ast||}=0.
			\end{align}
			(Note, in particular, if $\boldsymbol{\delta}(\mathbf{x})\equiv0$ then \eqref{eq:C.1.2} is satisfied.)
			\item[\textbf{C2:}] The weight sequence $\{\alpha_t\}_{t\geq0}$ is of form
			\begin{align}\label{eq:C.2}
				\alpha_t=\frac{a}{t+1}, \forall t\geq0,
			\end{align}
			where $a>0$ is a constant (note that \textbf{C2} implies \textbf{B5}).
			\item[\textbf{C3:}] Let $\mathbf{I}$ be the $l\times l$ identity matrix and $a, \mathbf{B}$ as in \eqref{eq:C.2} and \eqref{eq:C.1.1}, respectively. Then, the matrix $\boldsymbol{\Sigma}=a\mathbf{B}+\frac{1}{2}\mathbf{I}$ is stable.
			\item[\textbf{C4:}] The entries of the matrices, $\forall t\geq0,$ $x\in\mathbb{R}^l,$
			\begin{align*}
				\mathbf{A}(t,\mathbf{x})=\mathbb{E}[\boldsymbol{\gamma}(t,\mathbf{x},\omega)\boldsymbol{\gamma}^\top(t,\mathbf{x},\omega)],
			\end{align*}
			are finite, and the following limit exists:
			\begin{align*}
				\lim\limits_{t\to\infty, \mathbf{x}\to\mathbf{x}^\ast}\mathbf{A}(t,\mathbf{x})=\mathbf{S}_0.
			\end{align*}
			\item[\textbf{C5:}] There exists $\epsilon>0$, such that
			\begin{align*}
				\lim\limits_{R\to \infty}\sup\limits_{||\mathbf{x}-\mathbf{x}^\ast||<\epsilon}\sup\limits_{t\geq0}\int\limits_{||\boldsymbol{\gamma}(t+1,\mathbf{x},\omega)||>R}||\boldsymbol{\gamma}(t+1,\mathbf{x},\omega)||^2dP=0
			\end{align*}
		\end{itemize}
		Let Assumptions B1--B5 hold for $\{\mathbf{x}(t)\}_{t\geq0}$ in \eqref{eq:SAtheorem}. Them, starting from an arbitrary initial state, the Markov process, $\{\mathbf{x}^t\}_{t\geq0}$, converges a.s. to $\mathbf{x}^\ast$. In other words,
		\begin{align*}
			\mathbf{P}[\lim\limits_{t\to\infty}\mathbf{x}^t=\mathbf{x}^\ast]=1.
		\end{align*}
		The normalized process, $\{\sqrt{t}(\mathbf{x}^t-\mathbf{x}^\ast)\}_{t\geq0},$ is asymptotically normal if, besides Assumptions B1--B5, Assumptions C1--C5 are also satisfied. In particular, as $t\to\infty$
		\begin{align}\label{eq:convindisti}
			\sqrt{t}(\mathbf{x}^t-\mathbf{x}^\ast)\Rightarrow\mathcal{N}(\mathbf{0},\mathbf{S}),
		\end{align}
		where $\Rightarrow$ denotes convergence in distribution (weak convergence). Also, asymptotic variance, $\mathbf{S}$, in \eqref{eq:convindisti} is
		\begin{align*}
			\mathbf{S}=a^2\int\limits_{0}^\infty e^{\boldsymbol{\Sigma} v}\mathbf{S}_0e^{\boldsymbol{\Sigma}^\top v}dv
		\end{align*}
	\end{theorem}
	\subsection*{B. Additional results on nonlinearity $\varphi$}
	
	We present some properties of the function $\varphi$ defined in~\eqref{eq:phi-def}. As it is stated in~\cite{Ourwork}, we can intuitively see function $\varphi$ as a convolution-like transformation of nonlinearity $\Psi:\mathbb{R}\to\mathbb{R}$, where the convolution is taken with respect to the {probability density} function $p$ of random value $w$. If $w$ is generated by the underlying probability space $(\Omega, \mathcal{F}, \mathbb {P})$, we have that expectation of 
	\begin{align}\label{eq:defv}
		v=\Psi(a+w) - \varphi(a)
	\end{align}
	is equal to zero, i.e., $\mathbb{E}[v]=0$. Here, the expectation is taken with respect to $\mathcal{F}$. Hence, for all $t=0,1,...,$ we have that expectation of both of the sequences $\boldsymbol{\zeta}^t$, $\boldsymbol{\eta}^t$ defined in~\eqref{eq:zeta} is equal to zero, due to the fact that communication noise $\boldsymbol{\xi}^t$ and observation noise $\mathbf{n}^t$, $t=0,1,...,$ are generated by underlying probability space.
	
	\noindent We have following Lemma (see~\cite{PolyakNL}, see also~\cite{Ourwork}).
	\begin{lemma} [\cite{PolyakNL}]
		\label{Lemma-Polyak}
		Consider function $\varphi$ in \eqref{eq:phi-def}, where function $\Psi:\mathbb{R}\to\mathbb{R}$, satisfies Assumption \ref{as:nonlinearity}. Then, the following holds:
		\begin{enumerate}
			\item ${\varphi}$ is odd;
			\item If $|\Psi(\nu)| \leq c_1,$ for any $\nu \in \mathbb R$, then $|{\varphi}(a)| \leq c_1^{\prime},$ for any $a \in \mathbb R$, for some
			$c_1^\prime>0$;
			\item ${\varphi}(a)$ is monotonically nondecreasing;
			\item ${\varphi}(a) > 0,$ for any $ a>0$.
			\item ${\varphi}$ is continuous at zero;
			\item ${\varphi}$ is differentiable at zero, with a strictly positive derivative at zero, equal to:
			\begin{equation}
				\begin{split}
					{\varphi}'(0) &= \sum_{i=1}^s \left(\Psi(\nu_i + 0) - \Psi(\nu_i - 0)\right)p(\nu_i) 
					+ \sum_{i=0}^s\int_{\nu_i}^{\nu_{i+1}}\Psi'(\nu)p(\nu)d\nu ,
				\end{split}
			\end{equation}
			where $\nu_i, i=1,...,s$ are points of discontinuity of $\Psi$ such that $\nu_0 = -\infty$ and $\nu_{s+1} = +\infty$, and we recall that $p(u)$ is the pdf of random variable $w$.
		\end{enumerate}
	\end{lemma}
	
	\noindent From Lemma~\ref{Lemma-Polyak}, we have that $\varphi(a)=0$ if and only if $a=0$. Moreover, there exists a function $\delta:\mathbb{R}\to\mathbb{R}$, which is continuous in the vicinity of zero, such that 
	\begin{align}\label{eq:defdelta}
		\varphi(a)=\varphi(0)+\varphi^\prime(0)a+\delta(a)=\varphi^\prime(0)a+\delta(a),
	\end{align}
	and $\lim\limits_{a\to0}\frac{\delta(a)}{a}=0.$
	
	\noindent We now prove boundedness of the function $\mathbf{r}(\cdot)$ in equation~\eqref{eq:boundedr}. If condition 2 of Lemma~\ref{Lemma-Polyak} is satisfied for both functions $\varphi_\mathrm{c}$ and $\varphi_\mathrm{o}$, then the right hand side of ~\eqref{eq:boundedr} would be lesser or equal to some positive constant $c$, which would led to $\left\|\mathbf{r}(\mathbf{x})\right\|^2 \leq c_1 (1+V(\mathbf{x})).$ Suppose now that condition 3 of Lemma~\ref{Lemma-Polyak} is satisfied for the function $\varphi_\mathrm{c}$, then there exists some positive constant $c_1$ such that 
	\begin{align*}
		\left\| \frac{b}{a}\mathbf{L}_{\boldsymbol{\varphi}_\mathrm{c}}(\mathbf{x}-\mathbf{1}_{N}\otimes \boldsymbol{\theta}^{\ast})\right\|^2&=\left(\frac{b}{a}\right)^2\sum\limits_{i=1}^{N}\left\| \sum\limits_{j\in\Omega_{i}}\boldsymbol{\varphi}_\mathrm{c}(\mathbf{x}_i-\mathbf{x}_j)\right\|^2\\&\leq\left(\frac{b}{a}\right)^2\sum\limits_{i=1}^{N} \sum\limits_{j\in\Omega_{i}}\left\| \boldsymbol{\varphi}_\mathrm{c}(\mathbf{x}_i-\mathbf{x}_j)\right\|^2
		\\&\leq \left(\frac{b}{a}\right)^2\sum\limits_{i=1}^{N} \sum\limits_{j\in\Omega_{i}} \left( c \left( 1+ \left\|\mathbf{x}_i-\mathbf{x}_j\right\|^2\right)\right)\\&\leq \left(\frac{b}{a}\right)^2\sum\limits_{i=1}^{N} \sum\limits_{j\in\Omega_{i}} \left( c \left( 1+ \left\|\mathbf{x}_i-\boldsymbol{\theta}^\star \right\|^2+\left\|\mathbf{x}_j-\boldsymbol{\theta}^\star \right\|^2\right)\right)\\
		&\leq c_1 (1+V(\mathbf{x})),
	\end{align*}
	since we have that $\left\|\mathbf{x}_i-\boldsymbol{\theta}^\star \right\|^2\leq ||\mathbf{x}-\mathbf{1}_N\otimes\boldsymbol{\theta}^\ast||^2=V(\mathbf{x})$ for all $i=1,2,...,N.$ If we assume that condition 3 of Lemma~\ref{Lemma-Polyak} is satisfied for the function $\varphi_\mathrm{o}$, we will get that 
	\begin{align*}
		\left\| \mathbf{H}^{\top}\boldsymbol{\varphi}_\mathrm{o}\left(     
		\mathbf{H}\left(\mathbf{x}-\left(\mathbf{1}_{N}\otimes \boldsymbol{\theta}^{\ast}\right)
		\right)\right)\right\|^2&\leq \left\| \mathbf{H}\right\|^2 \left\|\boldsymbol{\varphi}_\mathrm{o}\left(     
		\mathbf{H}\left(\mathbf{x}-\left(\mathbf{1}_{N}\otimes \boldsymbol{\theta}^{\ast}\right)
		\right)\right)\right\|^2\\&\leq \left\| \mathbf{H}\right\|^2  c \left( 1+ \left\|\mathbf{H}\left(\mathbf{x}-\left(\mathbf{1}_{N}\otimes \boldsymbol{\theta}^{\ast}\right)
		\right)\right\|^2\right)\\&\leq  \left\| \mathbf{H}\right\|^2  c \left( 1+ \left\| \mathbf{H}\right\|^2 \left\|\mathbf{x}-\mathbf{1}_{N}\otimes \boldsymbol{\theta}^{\ast}\right\|^2\right).
	\end{align*}
	Therefore, $\left\| \mathbf{H}^{\top}\boldsymbol{\varphi}_\mathrm{o}\left(     
	\mathbf{H}\left(\mathbf{x}-\left(\mathbf{1}_{N}\otimes \boldsymbol{\theta}^{\ast}\right)\leq 
	\right)\right)\right\|^2\leq c_1 (1+V(\mathbf{x})),$ for some positive constant $c_1.$ Hence, inequality in~\eqref{eq:boundedr} is proven.
	
	\noindent Next we prove boundedness of $\mathbb{E}\left[\left\|\boldsymbol{\gamma}(t+1, \mathbf{x}^{t}, \omega)\right\|^2\right]$ in~\eqref{eq:boundedg}. If the function $\Psi$ in~\eqref{eq:phi-def} satisfies condition 5' of Assumption~\ref{as:nonlinearity}, whether
	$w$ in~\eqref{eq:phi-def} has finite or infinite variance, $v$ in~\eqref{eq:defv} is bounded, i.e.,
	\begin{align*}
		|v|^2\leq |\Psi(a+w)|^2 +| \varphi(a)|^2\leq c,
	\end{align*}
	for some positive constant $c$. If the function $\Psi$ in~\eqref{eq:phi-def} satisfies condition 5 of Assumption~\ref{as:nonlinearity} and $w$ has finite variance, we get that variance of $v$ in~\eqref{eq:defv} is bounded with $c\,(1+|a|^2)$ for some positive constant $c$, i.e.,
	\begin{align*}
		\mathbb{E}[|v|^2]&\leq\mathbb{E}[|\Psi(a+w)|^2 +| \varphi(a)|^2]\leq \mathbb{E}[c_1(1+|a+w|^2)+c_1'(1+|a|^2)]\\
		&\leq c_1(1+|a|^2+\mathbb{E}[|w|^2])+c_1'(1+|a|^2)\leq c\,(1+|a|^2),
	\end{align*}
	where $c_1$ and $c_2$ are some positive constants. Thus, whether condition 5 or 5' is satisfied for the function $\Psi$ in~\eqref{eq:phi-def}, variance of $v$ in~\eqref{eq:defv} is bounded with $c\,(1+|a|^2)$ for some positive constant $c$. Hence, we have that for $\boldsymbol{\zeta}^t$, $\boldsymbol{\eta}^t$ defined in~\eqref{eq:zeta}
	\begin{align*}
		\mathbb{E}[\boldsymbol{\zeta}^t]&\leq c'(1+V(\mathbf{x}))\\
		\mathbb{E}[\boldsymbol{\eta}^t]&\leq c''(1+V(\mathbf{x})),
	\end{align*}
	for all $t=0,1,...$, where $c'$ and $c''$ are some positive constants.

	\subsection*{C. Mutually dependent observation noise and mutually dependent communication noise}
	
	In this subsection we relax assumptions on observation and communication noises and show that Theorems~\ref{theorem-almost-surely} and~\ref{theorem-asymptotic-normality} continue to hold. We let Assumptions 1--6 still hold except those which overlap with the following generalizations:
	
	\begin{itemize}
		\item The observation noise $\mathbf{n}^t$ has the joint {probability density function $p_{\mathrm{o}}$} such that:
		\begin{align*}
        \int\limits_{\mathbf{a}\in\mathbb{R}^N}\|\mathbf{a}\|{p_\mathrm{o}(\mathbf{a})d\textbf{a}}<\infty, \quad \int\limits_{\mathbf{a}\in\mathbb{R}^N}\mathbf{a}\,{p_\mathrm{o}(\mathbf{a})d\textbf{a}},
		\end{align*}
		and $p_\mathrm{o}(\mathbf{a})=p_\mathrm{o}(-\mathbf{a})$, for all $\mathbf{a}\in\mathbb{R}^{N}$.
		\item A (possibly) different nonlinear function $\Psi_{\mathrm{o},i}:\mathbb{R}\to\mathbb{R}$ is assigned to each agent $i$. Each function $\Psi_{\mathrm{o},i}$ obeys Assumption \ref{as:nonlinearity}.
		\item The communication noise $\boldsymbol{\xi}_{ij}^t$ has the joint {probability density function $p_{\mathrm{c},ij}$} such that:
		\begin{align*} 
			\int\limits_{\mathbf{a}\in\mathbb{R}^M}\|\mathbf{a}\|{p_{\mathrm{c},ij}(\mathbf{a})d\textbf{a}}<\infty, \quad \int\limits_{\mathbf{a}\in\mathbb{R}^M}\mathbf{a}\,{p_{\mathrm{c},ij}(\mathbf{a})d\textbf{a}}=0,
		\end{align*}
		and {$p_{\mathrm{c},ij}(\mathbf{a})=p_{\mathrm{c},ij}(-\mathbf{a})$}, for all $\mathbf{a}\in\mathbb{R}^{M}$.
		\item A different nonlinear function $\Psi_{\mathrm{c},ij,\ell}:\mathbb{R}\to\mathbb{R}$ is assigned to each arc $(i,j)\in E_d$ and to each element $\ell=1,...,M$ of the communication noise $[\boldsymbol{\xi}_{ij}^t]_\ell$. Each function $\Psi_{\mathrm{c},ij,\ell}$ obeys Assumption \ref{as:nonlinearity}.
	\end{itemize}
	
	This means that observation noises of agents $i$ and $j$ can be mutually dependent. Moreover, the communication noises $\boldsymbol{\xi}_{ij}^t$ may have mutually dependent elements $[\boldsymbol{\xi}_{ij}^t]_\ell$, for $\ell=1,...,M$. Further, here, for simplicity, we assume that observation and communication noises are mutually independent.
	
	\noindent Let us define functions $\varphi_{\mathrm{o},i}:\mathbb{R}\to\mathbb{R}$ for $i=1,2,...,N$ and $\varphi_{ij,\ell}:\mathbb{R}\to\mathbb{R}$ for $(i,j)\in E$ and $\ell=1,2,...,M$ in the same manner as in~\eqref{eq:phi-def}, i.e.,
	
	\begin{align}
		\varphi_{\mathrm{o},i}(a)&=\int \Psi_{\mathrm{o},i}(a+w) {p_{\mathrm{o},i}(w)dw},\label{eqn:phiobs}\\
		\varphi_{\mathrm{c},ij,\ell}(a)&=\int \Psi_{\mathrm{c},ij,\ell}(a+w) {p_{\mathrm{c},ij,\ell}(w)dw}.\label{eqn:phicom}
	\end{align}
	Here, {$p_{\mathrm{o},i}$ and $p_{\mathrm{c},ij,\ell}$ are the marginal probability density functions }of random variables $\mathbf{n}^t_i$ and $[\boldsymbol{\xi}_{ij}^t]_\ell$, respectively. Following same steps as in the proofs of Theorems~\ref{theorem-almost-surely} and~\ref{theorem-asymptotic-normality}, almost sure convergence and asymptotic normality can be shown. In the following, we emphasize only differences. First of all, algorithm~\eqref{eq:algfin} gets replaced by
	
	\begin{align*}
		\mathbf{x}^{t+1}=\mathbf{x}^{t}-\alpha_t\left( \frac{b}{a}\hat{\mathbf{L}}_{\boldsymbol{\varphi}_\mathrm{c}}(\mathbf{x}^t)- \mathbf{H}^{\top}\boldsymbol{\varphi}_\mathrm{o}\left(     
		\mathbf{H}\left(\left(\mathbf{1}_{N}\otimes \boldsymbol{\theta}^{\ast}\right)
		-\mathbf{x}^{t}\right)\right) - \mathbf{H}^{\top}\boldsymbol{\zeta}^t+\frac{b}{a}\boldsymbol{\eta}^t \right).
	\end{align*}
	Now, the map $\hat{\mathbf{L}}_{\boldsymbol{\varphi}_\mathrm{c}}:\mathbb{R}^{MN}\to\mathbb{R}^{MN}$ is 
	\begin{align*}
		\hat{\mathbf{L}}_{\boldsymbol{\varphi}_\mathrm{c}} (\mathbf{x}) =\begin{bmatrix}
			\vdots\\
			\sum\limits_{j\in\Omega_{i}}\boldsymbol{\varphi}_{\mathrm{c},ij}(\mathbf{x}_i-\mathbf{x}_j)\\
			\vdots
		\end{bmatrix},
	\end{align*}
	for any $\mathbf{x}\in\mathbb{R}^{MN},$ where for all $(i,j)\in E,$ function $\boldsymbol{\varphi}_{\mathrm{c},ij}:\mathbb{R}^{M}\to\mathbb{R}^{M}$ is given with $\boldsymbol{\varphi}_{\mathrm{c},ij}(\mathbf{y}_1, \mathbf{y}_2,...,\mathbf{y}_{M})=[\varphi_{\mathrm{c},ij,1}(\mathbf{y}_1), \varphi_{\mathrm{c},ij,2}(\mathbf{y}_2),...,\varphi_{\mathrm{c},ij,M}(\mathbf{y}_{M})]^\top,$ for $\mathbf{y}\in\mathbb{R}^{M},$ functions $\varphi_{\mathrm{c},ij,\ell}(a)$ for $(i,j)\in E$ and $\ell=1,2,...,M$ are given by~\eqref{eqn:phicom}. Moreover, for $\mathbf{y}\in\mathbb{R}^{N},$ the map $\boldsymbol{\varphi}_\mathrm{o}:\mathbb{R}^{N}\to\mathbb{R}^{N}$ is now given with  $\boldsymbol{\varphi}_{\mathrm{o}}(\mathbf{y}_1, \mathbf{y}_2,...,\mathbf{y}_{N})=[\varphi_{\mathrm{o},1}(\mathbf{y}_1), \varphi_{\mathrm{o},2}(\mathbf{y}_2),...,\varphi_{\mathrm{o},N}(\mathbf{y}_{N})]^\top.$ Using the same notation, sequences $\boldsymbol{\zeta}^t\in\mathbb{R}^{N}$ and  $\boldsymbol{\eta}^t\in\mathbb{R}^{MN}$ are appropriate versions of the sequences defined in~\eqref{eq:zeta}. If we define quantities $\hat{\mathbf{r}}(\mathbf{x})$ and $\hat{\boldsymbol{\gamma}}(t+1,\mathbf{x},\omega)$ as follows
	\begin{align}
		\hat{\mathbf{r}}(\mathbf{x})&= -\frac{b}{a}\hat{\mathbf{L}}_{\boldsymbol{\varphi}_\mathrm{c}}(\mathbf{x})- \mathbf{H}^{\top}\boldsymbol{\varphi}_\mathrm{o}\left(     
		\mathbf{H}\left(\mathbf{x}-\left(\mathbf{1}_{N}\otimes \boldsymbol{\theta}^{\ast}\right)
		\right)\right),\label{eqn:rk}\\
		\hat{\boldsymbol{\gamma}}(t+1,\mathbf{x},\omega)&=-\frac{b}{a}\boldsymbol{\eta}^t+\mathbf{H}^{\top}\boldsymbol{\zeta}^t, \label{eqn:gammak}
	\end{align}
	it is easy to see that all conditions B1--B5 and C1--C5 from Theorem~\ref{theorem-SA} still hold (see~\cite{Ourwork}). The only difference occurs in the asymptotic covariance matrix $\mathbf{S}$, i.e., in $\mathbf{S}_0$, which is now given by
	\begin{align*}
		\mathbf{S}_0=\frac{b^2}{a^2}\mathbf{K}_{\boldsymbol{\eta}}+\mathbf{H}^\top \mathbf{K}_{\boldsymbol{\zeta}} \mathbf{H},
	\end{align*}
	where $\mathbf{K}_{\boldsymbol{\eta}}\in\mathbb{R}^{N\times N}$ and $\mathbf{K}_{\boldsymbol{\zeta}}\in\mathbb{R}^{MN\times MN}$ are the effective covariance matrices of communication and observation noises after passing through the appropriate nonlinearities (analogously defined as cross-covariance matrix $\mathbf{K}_\mathrm{c,o}$ in Theorem~\ref{theorem-asymptotic-normality}).

	\subsection*{D. Heavy-tailed noise and identity function}\label{subsection:HTUN}
	In this subsection, we show that the algorithm~\eqref{eq:alg2} does not converge in the presence of heavy-tailed observation and communication noise if at least one of the nonlinearities $\Psi_{\mathrm{o}}$ and $\Psi_{\mathrm{c}}$ is the identity function. This means that in the presence of heavy-tailed observation and communication noises, the algorithms from~\cite{Ourwork,KMR} do not converge, in fact, they exhibit an infinite variance solution sequence.
	
	\begin{theorem}[Infinite variance]
		\label{theorem-infinite-variance}
		For the sequence of iterates $\{\mathbf{x}^t\}, t=1,2,...,$ generated by~\eqref{eq:alg2}, we have that $\mathbb{E}[\|\mathbf{x}^t-\mathbf{1}_N\otimes\boldsymbol{\theta}^\star\|^2]=\infty, t=1,2,...,$ if at least one of the following statements is true.
		\begin{enumerate}
			\item Function $\Psi_{\mathrm{o}}$ is the identity function, i.e., $\Psi_{\mathrm{o}}(a)=a$ and the observation noise has infinite variance, i.e., $\int a^2d\Phi_{\mathrm{o}}=+\infty.$
			\item Function $\Psi_{\mathrm{c}}$ is the identity function, i.e.,  $\Psi_{\mathrm{c}}(a)=a$ and the communication noise has infinite variance, i.e., $\int a^2d\Phi_{\mathrm{c}}=+\infty.$
		\end{enumerate}
	\end{theorem}
	\begin{proof}
		For simplicity, we assume that if statement~1 holds there is no communication noise, i.e. $\boldsymbol{\xi}_{ij}\equiv0$ for all $(i,j)\in E_d$ , and \textit{vice versa},  if statement~2 holds we assume that there is no observation noise, i.e., $\mathbf{n}\equiv0$. If statement~1 holds, in the absence of communication noise, the algorithm~\eqref{eq:algcomp} can be written as
		\begin{align*}
			\mathbf{x}^{t+1}&=\mathbf{x}^{t}-\alpha_t\left( \frac{b}{a}\mathbf{L}_{\boldsymbol{\Psi}_\mathrm{c}}(\mathbf{x})- \mathbf{H}^{\top}\left(\mathbf{z}^{t}-\mathbf{H}\mathbf{x}^{t}\right)\right)\\
			&=\mathbf{x}^{t}-\alpha_t\left( \frac{b}{a}\mathbf{L}_{\boldsymbol{\Psi}_\mathrm{c}}(\mathbf{x})- \mathbf{H}^{\top}\left(\mathbf{H}(\mathbf{1}\otimes\boldsymbol{\theta}^\star)+\mathbf{n}^t-\mathbf{H}\mathbf{x}^{t}\right)\right).
		\end{align*}
		If we define $\mathbf{e}^t=\mathbf{x}^t-\mathbf{1}_N\otimes\boldsymbol{\theta}^\star$, $t=1,2,...,$ we have that $\mathbf{e}^{t+1}=\mathbf{F}^t(\mathbf{e}^t)+\alpha_t\mathbf{H}^\top \mathbf{n}^t,$
		where function $\mathbf{F}^t:\mathbb{R}^{MN}\to\mathbb{R}^{MN}$ is given by 
		$\mathbf{F}^t( \mathbf{y})=(\mathbf{I} + \alpha_t\mathbf{H}^\top\mathbf{H})\mathbf{y} -\alpha_t\frac{b}{a}\mathbf{L}_{\boldsymbol{\Psi}_\mathrm{c}}(\mathbf{y}),$ for $\mathbf{y}\in\mathbb{R}^{MN}.$ Therefore, we have that
		\begin{align*}
			\|\mathbf{e}^{t+1}\|^2&=\|\mathbf{F}^t(\mathbf{e}^t)\|^2+2\alpha_t (\mathbf{F}^t(\mathbf{e}^t))^\top \mathbf{H}^\top \mathbf{n}^t + \alpha_t^2\|\mathbf{H}^\top \mathbf{n}^t\|^2\\
			&\geq 2\alpha_t (\mathbf{H}\,\mathbf{F}^t(\mathbf{e}^t))^\top \mathbf{n}^t + \alpha_t^2\|\mathbf{H}^\top \mathbf{n}^t\|^2,
		\end{align*}
		and using the fact that $\mathbf{e}^t$ and $\mathbf{n}^t$ are independent, we have that 
		\begin{align*}
			\mathbb{E}[\|\mathbf{e}^{t+1}\|^2]\geq \alpha_t^2 \mathbb{E}[ \|\mathbf{H}^\top \mathbf{n}^t\|^2]=\infty,
		\end{align*}
		which completes the proof of statement~1.
		Proof of statement~2 follows directly from Appendix~B in~\cite{Ourwork}.     
	\end{proof}
	
	\subsection*{E. Hypothetical variant of algorithm from~\cite{F}}
	Firstly, we give an overview of algorithm that is proposed in~\cite{F}, for more information see~\cite{F}. They considered a network of $N$ agents where each agent $i=1,2,...,N$ at each time $t\geq0$ collects a linear transformation of unknown vector parameter $\mathbf{w}^0\in\mathbb{R}^M$ corrupted by noise as follows
	\begin{align*}
		d_i(t)=\mathbf{u}_{i,t} \mathbf{w}^0+ v_i(t),
	\end{align*}
	where $\mathbf{u}_{i,t}\in\mathbb{R}^M$ is a row regression vector and $v_i(t)\in\mathbb{R}$ is wide-sense stationary zero-mean impulsive noise process with variance $\sigma_{v,i}^2$. They introduced an agent-dependent and time-varying error nonlinearity, $h_{i,t}(e_i(t)),$ into the adaptation step and proposed following algorithm
	\begin{equation}\label{eq:algAliSayed}
		\begin{split}
			\boldsymbol{\psi}_{i,t}&=\mathbf{w}_{i,t-1}+\mu_i \mathbf{u}_{i,t}^\top h_{i,t}(e_i(t)),\\
			\mathbf{w}_{i,t}&=\sum\limits_{\ell \in \mathcal{N}_i}a_{\ell i}\boldsymbol{\psi}_{\ell,t},
		\end{split}
	\end{equation} 
	where $\mu_i$ is a step size parameter, $\mathcal{N}_i$ is the set of agents connected to agent $i$ including himself and $a_{\ell i}$ are weighting coefficients. For the error nonlinearity $h_{i,t}(e_i(t)),$ they set to be a linear combination of $B_i\geq1$ preselected sign-preserving basis functions, i.e., $h_{i,t}(e_i(t))=\boldsymbol{\alpha}_{i,t}^\top \boldsymbol{\varphi}_{i,t}(e_i(t)).$ As it is said in~\cite{F}, if agent $i$ were to run the sand-alone counterpart of the adaptive filter in~\eqref{eq:algAliSayed}, then the optimal nonlinearity that minimizes $i$-th agent MSE is given by $h^{\mathrm{opt}}_{i,t}(x)=-\frac{p_e'(x)}{p_e(x)}$ in terms of the pdf of the error signal. \\
	\noindent Even though the pdf is not available in practice, for the purpose of comparing algorithms in the specific numerical example when we know pdf, we introduce hypothetical variant of algorithm, by finding optimal $\boldsymbol{\alpha}_{i,t}^{\mathrm{opt}},$ for each agent $i$ at each time $t$, i.e., $\boldsymbol{\alpha}_{i,t}^{\mathrm{opt}}=\argmin\limits_{\boldsymbol{\alpha}_{i,t}} \mathbb{E}[h^{\mathrm{opt}}_{i,t}(e_i(t))-h_{i,t}(e_i(t)) ]^2$
	
	\subsection*{F. Derivations and numerical illustrations for Example 1}
	
	Derivation for the average per-agent asymptotic variance $\sigma_{{B}}^2=\frac{1}{N}\Tr(\mathbf{S})$ follows
	\begin{align*}
		\sigma_{{B}}^2&=\frac{1}{N}\Tr(a^2\int\limits_{0}^{+\infty}e^{\Sigma v}\mathbf{S}_0e^{\Sigma v}dv)
		=\frac{1}{N}a^2\sigma_\mathrm{o}^2h^2\int\limits_{0}^{+\infty}\Tr(e^{2\Sigma v}dv)\\
		&=\frac{1}{N}a^2\sigma_\mathrm{o}^2h^2\int\limits_{0}^{+\infty}N e^{(1-2ah^2\varphi_{\mathrm{o}}'(0))v}dv=\frac{a^2\sigma_\mathrm{o}^2h^2}{2ah^2\varphi_{\mathrm{o}}'(0)-1}.
	\end{align*}
	Integral in the last equality converge for $a>\frac{1}{2h^2\varphi_{\mathrm{o}}'(0)}.$\\
	\noindent If $a=a(B)=\frac{1}{2h^2\varphi_{\mathrm{o}}'(0)(B)}+\epsilon,$ for some constant $\epsilon>0$, we have that
	\begin{align*}
		\sigma_{{B}}^2&=\frac{\left(\frac{1}{2h^2\varphi_{\mathrm{o}}'(0)}+\epsilon\right)^2\sigma_\mathrm{o}^2h^2}{2\left(\frac{1}{2h^2\varphi_{\mathrm{o}}'(0)}+\epsilon\right)h^2\varphi_{\mathrm{o}}'(0)-1}=\frac{\left(\frac{1+2h^2\varphi_{\mathrm{o}}'(0)\epsilon}{2h^2\varphi_{\mathrm{o}}'(0)}\right)^2\sigma_\mathrm{o}^2h^2}{2\left(\frac{1}{2h^2\varphi_{\mathrm{o}}'(0)}+\epsilon\right)h^2\varphi_{\mathrm{o}}'(0)-1}\\
		&=\frac{\left(\frac{1+2h^2\varphi_{\mathrm{o}}'(0)\epsilon}{2h^2\varphi_{\mathrm{o}}'(0)}\right)^2\sigma_\mathrm{o}^2h^2}{1+2h^2\varphi_{\mathrm{o}}'(0)\epsilon-1}=\frac{\left(1+2h^2\varphi_{\mathrm{o}}'(0)\epsilon\right)^2\sigma_\mathrm{o}^2}{8h^4\varphi_{\mathrm{o}}'(0)^3\epsilon}.
	\end{align*}
	Next, we validate that $\lim\limits_{B\to0^+}\sigma_{{B}}^2=+\infty.$ It is suffice to show that $\lim\limits_{B\to0^+}\frac{\sigma_\mathrm{o}^2}{\varphi_{\mathrm{o}}'(0)^3}=+\infty,$ since $\sigma_B^2=\frac{\sigma_\mathrm{o}^2}{8h^4\varphi_{\mathrm{o}}'(0)^3\epsilon}+\frac{4h^2\epsilon\sigma_\mathrm{o}^2}{8h^4\varphi_{\mathrm{o}}'(0)^2\epsilon}+\frac{4h^4\epsilon^2\sigma_\mathrm{o}^2}{8h^4\varphi_{\mathrm{o}}'(0)\epsilon}.$ 
	\begin{align*}
		\lim\limits_{B\to0^+}\frac{\sigma_\mathrm{o}^2}{\varphi_{\mathrm{o}}'(0)^3}&=\lim\limits_{B\to0^+}\frac{B^2\int\limits_{-\infty}^{+\infty} \tanh^2(\frac{w}{B})f(w)dw}{\left(\int\limits_{-\infty}^{+\infty}\frac{1}{\cosh^2(\frac{w}{B})}f(w)dw\right)^3}=[\frac{w}{B}=t,dw=dt]\\
		&=\lim\limits_{B\to0^+}\frac{B^2\int\limits_{-\infty}^{+\infty} \tanh^2(\frac{w}{B})f(w)dw}{B^3\left(\int\limits_{-\infty}^{+\infty}\frac{1}{\cosh^2(w)}f(Bw)dw\right)^3}\\&=\lim\limits_{B\to0^+}\frac{\int\limits_{-\infty}^{+\infty} \tanh^2(\frac{w}{B})f(w)dw}{B\left(\int\limits_{-\infty}^{+\infty}\frac{1}{\cosh^2(w)}f(Bw)dw\right)^3}=+\infty,
	\end{align*}
	since $\lim\limits_{B\to0^+}\int\limits_{-\infty}^{+\infty} \tanh^2(\frac{w}{B})f(w)dw=1$ and $\lim\limits_{B\to0^+}=\int\limits_{-\infty}^{+\infty}\frac{1}{\cosh^2(w)}f(Bw)dw<+\infty.$\\
	We now prove that both of  the functions $\sigma_\mathrm{o}^2$ and $\varphi_{\mathrm{o}}'(0)$ are increasing function with respect to $B$. Suppose that $B_1<B_2,$ then we have that
	\begin{align*}
		B_1^2\tanh^2(\frac{w}{B_1})&<B_2^2\tanh^2(\frac{w}{B_2}),\\
		\frac{1}{\cosh^2(\frac{w}{B_1})}&<	\frac{1}{\cosh^2(\frac{w}{B_2})},
	\end{align*}
	for all $w\in\mathbb{R}.$ Moreover, since $f(w)\geq0$ for all $w\in\mathbb{R}$, we have that
	\begin{align*}
		B_1^2\tanh^2(\frac{w}{B_1})f(w)&<B_2^2\tanh^2(\frac{w}{B_2})f(w),\\
		\frac{1}{\cosh^2(\frac{w}{B_1})}f(w)&<	\frac{1}{\cosh^2(\frac{w}{B_2})}f(w),
	\end{align*}
	for all $w\in\mathbb{R}.$ Therefore, we have that
	\begin{align*}
		\sigma_\mathrm{o}^2(B_1) = \int\limits_{-\infty}^{+\infty}  B_1^2\tanh^2(\frac{w}{B_1})f(w)dw &< \int\limits_{-\infty}^{+\infty}  B_2^2\tanh^2(\frac{w}{B_2})f(w)dw =\sigma_\mathrm{o}^2(B_2),\\
		\varphi_{\mathrm{o}}'(0)(B_1)= \int\limits_{-\infty}^{+\infty} \frac{1}{\cosh^2(\frac{w}{B_1})}f(w)dw&<\int\limits_{-\infty}^{+\infty} \frac{1}{\cosh^2(\frac{w}{B_2})}f(w)dw=\varphi_{\mathrm{o}}'(0)(B_2).
	\end{align*}
	\noindent We now compare,  in the presence of heavy-tailed observation noise with pdf as in~\eqref{eq:htdistribution} for $\beta=2.05$, the proposed algorithm~\eqref{eq:algExp1} for the optimal choice of $B^\star$ with the method from~\cite{F} and its hypothetical variant (see Appendix E). For those methods we set that $B_i=2$, $\phi_{i,1}(x)=x$ and $\phi_{i,2}(x)=\tanh(x)$ for all agents. Furthermore, we set that weighting coefficients are chosen according to $a_{ij}=\frac{\tilde{\mathbf{A}}_{ ij}}{  \sum\limits_{\ell \in \mathcal{N}_i} \tilde{\mathbf{A}}_{\ell i}},$ where $\tilde{\mathbf{A}}=\mathbf{A}+\mathbf{I}.$ Moreover, for the smoothing recursions, zero initial conditions are assumed, $\nu_i$ is set to $0.9$ for every agent $i$ and $\epsilon=10^{-2}.$
	
	Figure \ref{fig:alisayedAn} shows Monte Carlo estimation of MSE for step size $\alpha_t=\frac{0.5}{t+1}$ and the Figure \ref{fig:alisayedA1} shows Monte Carlo estimation of MSE for step size $\alpha_t=\frac{1}{t+1}$. As it can be seen, the hypothetical variant of the method from~\cite{F} outperforms the proposed one in both of the scenarios. However, that is because with the hypothetical variant of~\cite{F} we optimize the choice of the nonlinearity for each agent at each time, whereas the proposed algorithm~\eqref{eq:algExp1} is optimized only by average per-agent asymptotic variance. Moreover, we see that the method from~\cite{F} is not as robust as the proposed algorithm~\eqref{eq:algExp1} with respect to the choice of the step size $\alpha_t$ (constant $a$).
	
	\begin{figure}[tbhp]
		\centering \subfloat[]{\label{fig:alisayedAn}\includegraphics[width=50mm]{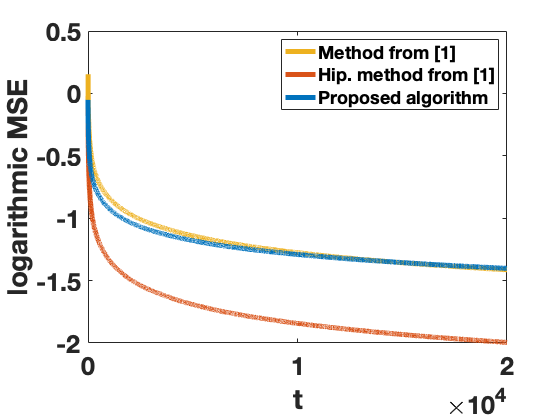}}
		\subfloat[]{\label{fig:alisayedA1}\includegraphics[width=50mm]{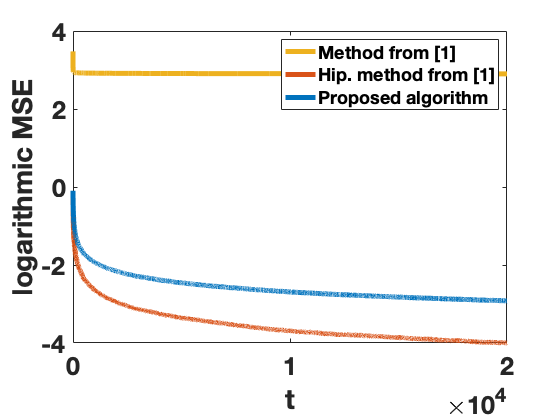}}
		\caption{(a) Monte Carlo-estimated per-sensor MSE error on logarithmic scale for the algorithm~\eqref{eq:algExp1} for optimal $B^\star$ and for algorithm and its hypothetical variant from~\cite{F} for $a=0.5$ (b) Monte Carlo-estimated per-sensor MSE error on logarithmic scale for the algorithm~\eqref{eq:algExp1} for optimal $B^\star$ and for algorithm and its hypothetical variant from~\cite{F} for $a=1$ }
	\end{figure}

  \subsection*{{G. Proof of the assertion in Remark~\ref{remark:meanestimator}}}
    {Here, we modify Theorem 3.1 from~\cite{Sub_Gaussian} and make it applicable to probability density functions that satisfy Assumption~\ref{as:observationnoise}. We will show that 
    \begin{align}\label{eq:subgaussian}
        \sup\limits_{p\in\mathcal{P}^M_{1+\epsilon}} \mathbb{P}\left(|\hat{\theta}_t-\theta^\star|>\left(\frac{8^{\frac{1}{\epsilon}}M^\frac{2}{\epsilon}\ln 2\delta}{t(\ln 2\delta-1)}\right)^{\frac{\epsilon}{1+\epsilon}}\right)\geq \delta,
    \end{align}
    for any $\theta^\star\in\mathbb{R}$, $\delta\in(0,\frac{1}{2})$, where $\mathcal{P}^M_{1+\epsilon}\subseteq\mathcal{P}$ denotes the subclass of all pdfs from $\mathcal{P}$ such that $1+\epsilon$-central moment equals $M$ for $\epsilon\in(0,1).$ Therefore, using Markov inequality, we get
    \begin{align*}
        \sup\limits_{p\in\mathcal{P}^M_{1+\epsilon}} t\mathbb{E}[|\hat{\theta}_t-\theta^\star|^2]\geq c_1 t^{\frac{1-\epsilon}{1+\epsilon}},
    \end{align*}
    for $c_1=\delta \left(\frac{8^{\frac{1}{\epsilon}}M^\frac{2}{\epsilon}\ln 2\delta}{\ln 2\delta-1}\right)^{\frac{2\epsilon}{1+\epsilon}}.$ Using that $\mathcal{P}^M_{1+\epsilon}\subseteq\mathcal{P}$ and taking the supremum with respect to $t$ we get \eqref{eq:subgassianestimator}.\\
    To show that~\eqref{eq:subgaussian} holds, we follow the same idea as in~\cite{Sub_Gaussian}.
   	Let us consider the class $\mathcal{P}_{+,-}=\{p_+,p_-\}$ of probability density function $p_+$ and $p_-$ such that $p_+$ and $p_-$ are probability density functions of uniform random variables on  $[\frac{p^2-p}{2},\frac{p^2+p}{2}]$ and on $[\frac{-p^2-p}{2},\frac{p-p^2}{2}]$, respectively, for $p\in(0,1).$ It is easy to see that means of probability density functions $p_+$ and $p_-$ are $\theta_+=\frac{p^2}{2}$ and $\theta_-=-\frac{p^2}{2}$, respectively. Moreover, $1+\epsilon$-th central moment of both pdfs is equal to
    \begin{align}\label{eq:1+epsilonMoment}
    M=\frac{p^{\epsilon+1}}{2^{\epsilon+1}(\epsilon+2)}.
    \end{align}
    Let $(X_j,Y_j), j=1,2,..,t$ be i.i.d. pairs random variables such that $p_+$ is pdf of $X_1,$ and $Y_1=X_1$ if $X_1\in I=[\frac{p^2-p}{2},\frac{p-p^2}{2}]$ and $Y_1=-X_1$ if $X_1\notin I.$ 
    Notice that probability density function of $Y_1$ is $p_-.$
    Since we have that $\mathbb{P}\{X_1\in I\}=1-p,$ for $X^t=(X_1,X_2,...,X_t)$ and $Y^t=(Y_1,Y_2,...,Y_t)$, we have that
    \begin{align*}
    \mathbb{P}\{X^t=Y^t\}=(1-p)^t.
    \end{align*}
 	Using that $1-p\geq e^{\frac{-p}{1-p}},$ we have that $\mathbb{P}\{X^t=Y^t\}=(1-p)^t\geq2\delta,$ if $p\leq \frac{\ln2\delta}{\ln2\delta-t}.$ Setting that $p:=\frac{\ln 2\delta}{t(\ln 2\delta -1)},$ we have that $p\in(0,1)$ for all $t=1,2,...$ and $\delta\in(0,\frac{1}{2}).$ Let $\hat{\theta}_t=\hat{\theta}_t(\cdot)$ be any estimator, then we have that
    \begin{align*}
        \max&\left(\mathbb{P}\Big\{|\hat{\theta}_t(X^t)-\theta_+|>\frac{p^2}{2}\Big\},\mathbb{P}\Big\{|\hat{\theta}_t(Y^t)-\theta_-|>\frac{p^2}{2}\Big\}\right)\\
        &\geq\frac{1}{2}\mathbb{P}\Big\{|\hat{\theta}_t(X^t)-\theta_+|>\frac{p^2}{2} \,\,\,or\,\,\, |\hat{\theta}_t(Y^t)-\theta_-|>\frac{p^2}{2}\Big\}\\
        &\geq\frac{1}{2}\mathbb{P}\{ \hat{\theta}_t(X^t)=\hat{\theta}_t(Y^t)\}\\
        &\geq\frac{1}{2}\mathbb{P}\{X^t=Y^t\}\geq\delta.
    \end{align*}
    Finally, using~\eqref{eq:1+epsilonMoment} we get that $\frac{\left(\frac{p^2}{2}\right)^{\frac{\epsilon+1}{2}}}{2\sqrt{2}}\geq\frac{\left(\frac{p^2}{2}\right)^{\frac{\epsilon+1}{2}}}{2^{\frac{\epsilon+1}{2}}(\epsilon+1)}=M\geq M p^{\frac{\epsilon}{2}},$ which gives us that 
    $\frac{p^2}{2}\geq\left(8^{\frac{1}{\epsilon}}M^{\frac{2}{\epsilon}}p\right)^{\frac{\epsilon}{\epsilon+1}}$ and therefore we have that
    \begin{align*}
        \max&\left(\mathbb{P}\Big\{|\hat{\theta}_t(X^t)-\theta_+|>
        \left(\frac{8^{\frac{1}{\epsilon}}M^\frac{2}{\epsilon}\ln 2\delta}{t(\ln 2\delta-1)}\right)^{\frac{\epsilon}{1+\epsilon}}\Big\}\right.,\\ &\left.\mathbb{P}\Big\{|\hat{\theta}_t(Y^t)-\theta_-|> \left(\frac{8^{\frac{1}{\epsilon}}M^\frac{2}{\epsilon}\ln 2\delta}{t(\ln 2\delta-1)}\right)^{\frac{\epsilon}{1+\epsilon}}\Big\}\right)\geq\delta.
    \end{align*}
    Since we have that $\mathcal{P}_{+,-}\subseteq\mathcal{P}^M_{1+\epsilon},$ it follows that~\eqref{eq:subgaussian} also holds.}

\subsection*{{H. Proof of extensions in Remark~\ref{remark:regressors}}}
	{For compact notation, we set that $\overline{\mathbf{H}}$ and $\widetilde{\mathbf{H}}^t$ are the $N\times(MN)$ matrices whose $i$-th row vectors are equal to $[\mathbf{0},...,\mathbf{0},(\overline{\mathbf{h}}_i)^\top,\mathbf{0},...,\mathbf{0}]$ and $[\mathbf{0},...,\mathbf{0},(\widetilde{\mathbf{h}}_i^t)^\top,\mathbf{0},...,\mathbf{0}]$, respectively. Hence, for $\mathbf{H}^t=\overline{\mathbf{H}}^t+\widetilde{\mathbf{H}}^t$, we have that~\eqref{eq:obsregressor} can be written, in compact form, as
	\begin{align}\label{eq:obsregressorComp}
	    \mathbf{z}^{t} = \mathbf{H}^t\left(\mathbf{1}_{N}\otimes \boldsymbol{\theta}^{\ast}\right)+\mathbf{n}^{t}=\overline{\mathbf{H}}\left(\mathbf{1}_{N}\otimes \boldsymbol{\theta}^{\ast}\right)+\widetilde{\mathbf{H}}^t\left(\mathbf{1}_{N}\otimes \boldsymbol{\theta}^{\ast}\right)+\mathbf{n}^{t}.
	\end{align}
	Under this setting, we modify algorithm~\eqref{eq:alg2} such that, at each time $t=0,1,...,,$, each agent $i$ updates its estimate $\mathbf{x}_i^t$ according to
    \begin{align}\label{eq:algReg}
		\mathbf{x}_{i}^{t+1}=\mathbf{x}_{i}^{t}-\alpha_{t}\left(\frac{b}{a}\sum_{j\in\Omega_{i}}
		\boldsymbol{\Psi}_\mathrm{c}\left( \mathbf{x}_{i}^{t}-\mathbf{x}_{j}^{t} 
		+\boldsymbol{\xi}_{ij}^t\right)-\overline{\mathbf{h}}_{i}{\Psi}_\mathrm{o}\left(z_{i}^{t}-\overline{\mathbf{h}}_{i}^{\top}\mathbf{x}_{i}^{t}\right)\right).
	\end{align}
	Assuming that all Assumptions \ref{as:newtworkmodelandobservability}-\ref{as:nonlinearity} still hold (except those which overlap and are hence replaced with assumptions in Remark~\ref{remark:regressors}), we show that the results in subsections~\ref{subsection:as},~\ref{subsection:an} and~\ref{subsection-MSE} continue to hold for  algorithm~\eqref{eq:algReg}. Following the same idea as in Section~\ref{section-theoresults}, we write algorithm~\eqref{eq:algReg}, in compact form, by:
    \begin{align}\label{eq:algRegComp}
		\mathbf{x}^{t+1}=\mathbf{x}^{t}-\alpha_t\left( \frac{b}{a}\mathbf{L}_{\boldsymbol{\Psi}_\mathrm{c}}(\mathbf{x})- \overline{\mathbf{H}}^{\top}\boldsymbol{\Psi}_\mathrm{o}\left(\mathbf{z}^{t}-\overline{\mathbf{H}}\mathbf{x}^{t}\right)\right).
	\end{align}
	Substituting~\eqref{eq:obsregressorComp} into~\eqref{eq:algRegComp}, we get that
     \begin{align*}
		\mathbf{x}^{t+1}&=\mathbf{x}^{t}-\alpha_t\left( \frac{b}{a}\mathbf{L}_{\boldsymbol{\Psi}_\mathrm{c}}(\mathbf{x})- \overline{\mathbf{H}}^{\top}\boldsymbol{\Psi}_\mathrm{o}\left(\overline{\mathbf{H}}\left(\mathbf{1}_{N}\otimes \boldsymbol{\theta}^{\ast}\right)+\widetilde{\mathbf{H}}^t\left(\mathbf{1}_{N}\otimes \boldsymbol{\theta}^{\ast}\right)+\mathbf{n}^{t}-\overline{\mathbf{H}}\mathbf{x}^{t}\right)\right)\\
        &=\mathbf{x}^{t}-\alpha_t\left( \frac{b}{a}\mathbf{L}_{\boldsymbol{\Psi}_\mathrm{c}}(\mathbf{x})- \overline{\mathbf{H}}^{\top}\boldsymbol{\Psi}_\mathrm{o}\left(\overline{\mathbf{H}}\left(\mathbf{1}_{N}\otimes \boldsymbol{\theta}^{\ast}-\mathbf{x}^{t}\right)+\widetilde{\mathbf{H}}^t\left(\mathbf{1}_{N}\otimes \boldsymbol{\theta}^{\ast}\right)+\mathbf{n}^{t}\right)\right).
	\end{align*}
    Recalling $\boldsymbol{\eta}^t\in\mathbb{R}^{MN}$ from~\eqref{eq:zeta} and defining 
    $\boldsymbol{\zeta}^t\in\mathbb{R}^{N}$ by 
    \begin{align*}
\boldsymbol{\zeta}^t=\boldsymbol{\Psi}_\mathrm{o}\left(\overline{\mathbf{H}}\left(\mathbf{1}_{N}\otimes \boldsymbol{\theta}^{\ast}-\mathbf{x}^{t}\right)+\widetilde{\mathbf{H}}^t\left(\mathbf{1}_{N}\otimes \boldsymbol{\theta}^{\ast}\right)+\mathbf{n}^{t}\right)-\boldsymbol{\varphi}_\mathrm{o}\left(\overline{\mathbf{H}}\left(\left(\mathbf{1}_{N}\otimes \boldsymbol{\theta}^{\ast}\right)
		-\mathbf{x}^{t}\right)\right)
    \end{align*}
    algorithm~\eqref{eq:algRegComp} can be written by
    \begin{align}
		\mathbf{x}^{t+1}=\mathbf{x}^{t}-\alpha_t\left( \frac{b}{a}\mathbf{L}_{\boldsymbol{\varphi}_\mathrm{c}}(\mathbf{x}^t)- \overline{\mathbf{H}}^{\top}\boldsymbol{\varphi}_\mathrm{o}\left(     
		\overline{\mathbf{H}}\left(\left(\mathbf{1}_{N}\otimes \boldsymbol{\theta}^{\ast}\right)
		-\mathbf{x}^{t}\right)\right) - \overline{\mathbf{H}}^{\top}\boldsymbol{\zeta}^t+\frac{b}{a}\boldsymbol{\eta}^t \right),
	\end{align}
    Since random variable $\widetilde{\mathbf{H}}^t\left(\mathbf{1}_{N}\otimes \boldsymbol{\theta}^{\ast}\right)+\mathbf{n}^{t}$ satisfies Lemma~\ref{Lemma-Polyak}, the rest of the proofs are same as in the Section~\ref{section-theoresults}.}

\end{document}